\documentclass[12pt]{amsart}

\usepackage{amssymb, todonotes, amsmath, amsfonts, amsxtra}

\usepackage[margin=2.5cm, a4paper]{geometry}
\usepackage{tikz}
\usetikzlibrary{matrix}
\usetikzlibrary{chains}

\newtheorem{thm}{Theorem}[section]
\newtheorem{prop}[thm]{Proposition} 
\newtheorem{cor}[thm]{Corollary} 

\newtheorem{lem}[thm]{Lemma}
\newtheorem{rem}[thm]{Remark}
\newtheorem{defn}[thm]{Definition}

\newtheorem{example}[thm]{Example}
\numberwithin{equation}{section}

\renewcommand{\MR}[1]{}

\newcommand{\cald}{\mathcal{D}}

\newcommand{\calt}{\mathcal{T}}
\newcommand{\calk}{\mathcal{K}}
\newcommand{\call}{\mathcal{L}}

\begin{document}
	
\title[Weighted Cuntz algebras]%
{Weighted Cuntz algebras}
%
\author[L. Helmer and B. Solel]{Leonid~Helmer \, and \, Baruch~Solel}

\address{Department of Mathematics, Ben Gurion University, Beer Sheva, Israel} \email{leonihe@gmail.com}
\address{Department of Mathematics, Technion - Israel Institute of Technology, Haifa 32000, Israel}
\email{mabaruch@technion.ac.il}

\date{}

%



\begin{abstract}
	We study the $C^*$-algebra $\mathcal{T}/\mathcal{K}$ where $\mathcal{T}$ is the $C^*$-algebra generated by $d$ weighted shifts on the Fock space of $\mathbb{C}^d$, $\mathcal{F}(\mathbb{C}^d)$, ( where the weights are given by a sequence $\{Z_k\}$ of matrices $Z_k\in M_{d^k}(\mathbb{C})$) and $\mathcal{K}$ is the algebra of compact operators on the Fock space. If $Z_k=I$ for every $k$, $\mathcal{T}/\mathcal{K}$ is the Cuntz algebra $\mathcal{O}_d$.
	
	We show that $\mathcal{T}/\mathcal{K}$ is isomorphic to a Cuntz-Pimsner algebra and use it to find conditions for the algebra to be simple. 
	
	We present examples of simple and of non simple algebras of this type. 
	
	We also describe the $C^*$-representations of $\mathcal{T}/\mathcal{K}$.
\end{abstract}

\maketitle

\thanks{{\it key words and phrases.} Weighted shift, Simplicity, Cuntz-Pimsner algebra,\\ $C^*$-Correspondence, Fock space, $C^*$-algebra.}


\section{Introduction }

In \cite{OD75} O'Donovan studied the $C^*$-algebra generated by a single weighted shift modulo the compact operators. To do this he first proved that such a $C^*$-algebra is isomorphic to a certain crossed product of a commutative $C^*$-algebra by an action of $\mathbb{Z}$. Then, he was able to use known results for such crossed products to study the algebra. The same approach was used by P. Ghatage in \cite{G} (and also, with W. Phillips, in \cite{GP}).

In \cite{MuS16}, the second author, with P. Muhly, introduced and studied algebras of weighted shifts on the Fock space associated with a correspondence. This is a far reaching generalization of the classical weighted shift (on $\ell_2$). The emphasis there was on the nonself adjoint operator algebras associated with such shifts (these algebras are called weighted Hardy algebras). 

Here we explore the case where the correspondence is a finite dimensional Hilbert space  $\mathbb{C}^d$. Recall that the Fock space associated with $\mathbb{C}^d$ is the Hilbert space
$$\mathcal{F}(\mathbb{C}^d):=\mathbb{C}\oplus \mathbb{C}^d \oplus (\mathbb{C}^d)^{\otimes 2} \oplus (\mathbb{C}^d)^{\otimes 3}\oplus \cdots .$$ So that, when $d=1$, we get $\ell_2$. Now, the sequence of weights is a sequence $\{Z_k\}_{k=0}^{\infty}$ of invertible  matrices where $Z_k$ is a $d^k\times d^k$ complex matrix, $Z_0=1$ and there are $0<\epsilon$ and $M\geq \epsilon$ such that $\epsilon I \leq |Z_k| \leq MI$ (for every $k$). As shown in Lemma~\ref{ueq} , we can (and often will) assume that $Z_k$ is positive for all $k$.

For every $1\leq i \leq d$, write $W_i$ for the operator, in $B(\mathcal{F}(\mathbb{C}^d))$, defined by
$$W_i (\xi_1 \otimes \xi_2 \otimes \cdots \otimes \xi_n)=Z_{n+1}(e_i \otimes \xi_1 \otimes \xi_2 \otimes \cdots \otimes \xi_n)$$ where each $\xi_j$ lies in $\mathbb{C}^d$ and $\{e_1,\ldots,e_d\}$ is the standard orthonormal basis of $\mathbb{C}^d$. We refer to $\{W_i\}$ as the weighted shifts and note that we can write $W_i=ZS_i$ where $S_i$ is the unweighted shift defined by $$S_i (\xi_1 \otimes \xi_2 \otimes \cdots \otimes \xi_n)=(e_i \otimes \xi_1 \otimes \xi_2 \otimes \cdots \otimes \xi_n)$$ and
$$Z:=Z_0\oplus Z_1 \oplus Z_2 \oplus \cdots .$$
Note also that, by our assumptions on the sequence $\{Z_k\}$, $Z$ and each $W_i$ is bounded and $Z$ is bounded and invertible. Such weighted shifts were studied by G. Popescu in \cite{Po} but in his study the matrices $Z_k$ were assumed to be diagonal.

Write $\mathcal{T}(\mathbb{C}^d,Z)$ (or, simply, $\mathcal{T}$) for the $C^*$-algebra generated by $\{W_i\}_{i=1}^{d}$. It turns out that this algebra contains the algebra of the compact operators on $\mathcal{F}(\mathbb{C}^d)$ (denoted $\mathcal{K}$) and we propose to study the $C^*$-algebra $\mathcal{T}/ \mathcal{K}$. If $Z_k=I$ for every $k$, $\mathcal{T}/\mathcal{K}$ is isomporphic to the Cuntz algebra $\mathcal{O}_d$. Thus, we  refer to the algebra $\mathcal{T}/\mathcal{K}$  as a \emph{weighted Cuntz algebra}.

If $d=1$ this is the algebra studied in \cite{OD75}, \cite{G} and \cite{GP}. As mentioned above, it was done by presenting it as a crossed product $C^*$-algebra (see \cite[Theorem 3.1.1]{OD75}) and then applying the theory of crossed product $C^*$-algebras. 

In our case, we will prove, in Theorem~\ref{isomorphism}, that we can present the algebra $\mathcal{T}/\mathcal{K}$ as a Cuntz-Pimsner algebra associated with a $C^*$-correspondence $q(F)$ over a $C^*$-algebra $q(\mathcal{D})$. Note that this generalizes the case $d=1$  since a $C^*$-crossed product is a Cuntz-Pimsner algebra. (It also generalizes Theorem 8.1 of \cite{MuS16} where $\mathbb{C}^d$ is replaced by the correspondence associated with an automorphic action on a von Neumann algebra). Then, at least in principle, one can apply the theory of Cuntz-Pimsner algebras to study the algebra $\mathcal{T}/\mathcal{K}$. There also seems to be a potential connection between our study here and the approach via Exel's type crossed products by endomorphisms (see \cite{Kw17}).

In this paper, we use this approach to study the question of simplicity of the algebra $\mathcal{T}/\mathcal{K}$ and its $C^*$-representations. 

The problem with this approach is that the correspondence that we get (or the dynamics in the case $d=1$) can be quite complicated and one may have to impose additional conditions on the sequence of weights. For $d=1$ the simplicity question was answered in \cite{G} for the case where the sequence of weights is almost periodic (see \cite[Proposition 2.2]{G}). Here, for the study of simplicity, we make the assumption that the sequence $\{Z_k\}$ is essentially periodic. We refer to this assumption as Condition A(p) (see the discussion preceding Lemma~\ref{uz}).

Assuming Condition A(p), we are able, in Theorem~\ref{simplicity}, to find conditions for the algebra $\mathcal{T}/\mathcal{K}$ to be simple. For this we use the characterization of simplicity of Cuntz-Pimsner algebras proved by J. Schweizer in \cite{Sch01} (see Theorem~\ref{simpleCP} below) and an analysis of the algebra $q(\mathcal{D})$ and its ideals. In fact, we show, in Corollary~\ref{UHFC00}, that $q(\mathcal{D})$ is isomorphic to  $M_{d^{\infty}}\otimes \mathcal{C}_{00}$ where $M_{d^{\infty}}$ is the UHF algebra obtained as a direct limit of the algebras $M_{d^n}(\mathbb{C})$ and $\mathcal{C}_{00}$ is given explicitely.

In Section~\ref{examples} we present some examples of simple and non simple algebras of the form $\mathcal{T}/\mathcal{K}$ for different sequences of weights (see also Corollary~\ref{EX}).

Finally, in Section 6, we describe the $C^*$-representations of the algebra $\mathcal{T}/\mathcal{K}$ and show that each such representation is given by an invertible operator $R$ and a Cuntz family $\{U_i\}$ satisfying certain conditions.

What we see in both Theorem~\ref{simplicity} and Theorem~\ref{representations} is that (assuming our Condition A(p)) understanding the structure of $\mathcal{T}/\mathcal{K}$ reduces to understanding the structure of a certain ``corner" of $q(\mathcal{D})$ (denoted $\mathcal{C}_{00}$). In some cases we know that $\mathcal{C}_{00}$ is finite dimentional and then we can get more definite results (as demonstrated in the examples that we present).

\textbf{Notation}

Let $\mathbb{F}_d^+$ be the free algebra on $d$ generators. An element $\alpha\in \mathbb{F}_d^+$ is a word $\alpha=\alpha_1 \alpha_2 \cdots \alpha_n$ where $\alpha_i\in \{1,\ldots,d\}$. The length of this word will be written $|\alpha|$. 

We shall often write $M_k$ for $M_k(\mathbb{C})$. For a matrix $A\in M_{d^n}$, the rows and columns will be indexed by words in $\mathbb{F}_d^+$. So we can write $A=(a_{\alpha,\beta})$ (with $|\alpha|=|\beta|=n$). We shall write $e_{\alpha,\beta}$ for the matrix unit with $1$ in the $\alpha,\beta$ position and $0$s elsewhere.

Also, given $u_1,u_2,\ldots,u_d$ and $\alpha\in \mathbb{F}_d^+$ with $|\alpha|=n\geq 1$, we write $u_{\alpha}=u_{\alpha_1} u_{\alpha_2} \ldots u_{\alpha_n}$. For the trivial word $\emptyset$, $u_{\emptyset}=I$.

\section{The weighted Cuntz algebras as Cuntz-Pimsner algebras}

Let $E=\mathbb{C}^d$ with $0<d<\infty$ be a finite dimensional Hilbert space. Then the (full) Fock space 
$\mathcal{F}(\mathbb{C}^d)=\bigoplus_{k=0}^{\infty}(\mathbb{C}^d)^{\otimes k}$ is a Hilbert space. 
For every $\xi\in \mathbb{C}^d$ we have $\xi=a_1{e_1}+...+a_d{e_d}$ where $\{{e_i}: i=1,...,d\}$ is the standard basis in $\mathbb{C}^d$. Hence, if $T_\xi$ is the creation operator associated with $\xi$, then $T_\xi=a_1T_{{e_1}}+...+a_dT_{{e_d}}$. For the sake of simplicity we shall write $S_i$ for  $T_{{e_i}}$, $i=1,...,d$. The Toeplitz algebra $\mathcal{T}(\mathbb{C}^d)$ of $\mathbb{C}^d$ is defined to be the algebra
$C^*(S_1,...,S_d)$.

The following definition appeared in \cite{MuS16} in the context of $W^*$-correspondences.

\begin{defn}  A sequence $Z=\{Z_k\}_{k\geq 0}$ of  operators $Z_k\in B((\mathbb{C}^d)^{\otimes k})$  will be called a weight sequence  in case:

1. $Z_0=I_A$,

2. $\sup\|Z_k\|<\infty$,


3. There is an $\epsilon >0$ such that $|Z_k|\geq \epsilon I$ for all $k\geq 1$ .
	
\end{defn}

Given a weight sequence it defines a weight operator $Z=diag(Z_1,Z_2,...) :\mathcal{F}(\mathbb{C}^d)\rightarrow \mathcal{F}(\mathbb{C}^d)$,
 where $Z_k:
(\mathbb{C}^d)^{\otimes k}\rightarrow (\mathbb{C}
^d)^{\otimes k}$. We write, for $1\leq i \leq d$, $W_i=ZS_i$ and refer to it as a weighted shift. The weighted Toeplitz algebra, $\mathcal{T}(\mathbb{C}^d,Z)$, is the $C^*$-algebra generated by $\{W_i: 1\leq i \leq d\}$. It follows from our assumptions on the weight sequence that the hypotheses $\textbf{A}$ and $\textbf{B}$ from \cite[Section 6]{MuS16} are satisfied and (using  \cite[Proposition 6.5]{MuS16}) the Toeplitz algebra $\mathcal{T}(\mathbb{C}^d)$ is a subalgebra of the weighted Toepliz algebra $\mathcal{T}(\mathbb{C}^d, Z)=C^*(ZS_i:i=1,...,d)$. It also follows that $\mathcal{T}(\mathbb{C}^d, Z)=C^*(ZS_i:i=1,...,d)$ contains the algebra $\mathcal{K}(\mathcal{F}(\mathbb{C}^d)$ of the compact operators on the Fock space.

\begin{rem}
	We assumed that $Z_k$ satisfy condition 3. above in order to use \cite[Proposition 6.5]{MuS16} and conclude that $\mathcal{T}(\mathbb{C}^d)$ is a subalgebra of the weighted Toepliz algebra $\mathcal{T}(\mathbb{C}^d, Z)$. However, the following lemma (which is well known when $d=1$) shows that we can assume that the operators $\{Z_k\}$ are positive and the structure of the algebra  $\mathcal{T}(\mathbb{C}^d, Z)$ (and its quotient by the compact operators) will not change.
	\end{rem}
\begin{lem}\label{ueq}
Let $\{Z_k\}$ is a sequence of operators $Z_k\in B((\mathbb{C}^d)^{\otimes k})$ that satisfy conditions 1.-3. above  and let $\mathcal{T}(\mathbb{C}^d, Z)=C^*(ZS_i:i=1,...,d)$ as above. Then there is a sequence $\{Z_k'\}$ of positive operators satisfying 1.-3. above such that the algebra $\mathcal{T}':=	C^*(Z'S_i:i=1,...,d)$ is unitarily equivalent to $C^*(ZS_i:i=1,...,d)$. 
	
	\end{lem}
\begin{proof}
We construct $Z_k'$ inductively using the fact that, given an invertible  operator $A$, there is a unitary operator $V$ such that $V^*A\geq 0$ (simply use the polar decomposition of the invertible operator $A$) . Let $Z'_0=I$. Since $Z_1$ is invertible, we have a unitary operator $U_1$ such that $Z_1':=U_1^*Z_1\geq 0$. Since $Z_2(I_{\mathbb{C}^d}\otimes U_1)$ is invertibe, there is a unitary operator $U_2$ such that $Z_2':=U_2^*(Z_2(I_{\mathbb{C}^d}\otimes U_1))\geq 0$. Continuing this way, we get a sequence of unitary operators $\{U_k:k\geq 0\}$ (with $U_0=I$) such that, for every $k$, $$Z_k':=U_{k}^*Z_k (I_{\mathbb{C}^d}\otimes U_{k-1})\geq 0.$$ The sequence $\{U_k\}$ defines a unitary operator $U\in B(\mathcal{F}(\mathbb{C}^d))$ such that $$U^*ZS_iU=Z'S_i$$ for all $i$.	
	
	\end{proof}

We write $P_n$ for the projection of the Fock space $\mathcal{F}(\mathbb{C}^d)$ onto $(\mathbb{C}^d)^{\otimes n}$. 
The following lemma is a consequence of the inclusion $\mathcal{T}(\mathbb{C}^d)\subseteq \mathcal{T}(\mathbb{C}^d, Z)$. 
\begin{lem}\label{T}
The operators $P_k$ ($0\leq k$), $S_i$ ($1\leq i \leq d$) and $Z$ are all contained in $\mathcal{T}(\mathbb{C}^d, Z)$. So, $\mathcal{T}(\mathbb{C}^d, Z)$ is the $C^*$-algebra generated by $\{I,S_i,Z\}$.
\end{lem}
\begin{proof}
Clearly $S_i \in \mathcal{T}(\mathbb{C}^d, Z)$ and, therefore, also $P_0=I-\Sigma S_iS_i^*$ and $P_k=\Sigma_{|\alpha|=k} S_{\alpha}P_0S_{\alpha}^*$. For $Z$, note that $ZP_0=P_0\in \mathcal{T}(\mathbb{C}^d, Z)$ and $ZP_0^{\perp}=\Sigma_i (ZS_i)S_i^* \in \mathcal{T}(\mathbb{C}^d, Z)$.
\end{proof}
 Write $W_t:=\Sigma_{n=0}^{\infty}e^{int}P_n$
and $\gamma_t=Ad (W_t)$ for $t\in \mathbb{R}$ to get the gauge automorphism group on $B(\mathcal{F}(\mathbb{C}^d))$. 



Next, we write
$$\mathcal{D}:=\{X\in \mathcal{T}(\mathbb{C}^d,Z):\; \gamma_t(X)=X\; \forall t \;\} $$ and $$ F:=\{X\in \mathcal{T}(\mathbb{C}^d,Z):\; \gamma_t(X)=e^{it}X\; \forall t\;\} .$$
For simplicity, we shall write $\mathcal{T}$ for $\mathcal{T}(\mathbb{C}^d, Z)$.

We shall need the following well known definitions.

\begin{defn}
	\begin{enumerate}
		\item [(a)] $E$ is a (right) Hilbert $C^*$-module over the $C^*$-algebra $A$ if $E$ is a right $A$-module equipped with an $A$-valued inner product (assumed to be $A$-linear in the second term) which is complete in the norm $\|\xi\|^2:=\|\langle \xi,\xi\rangle\|$.
		\item[(b)] An operator $T:E\rightarrow E$ on a Hilbert $C^*$-module is said to be adjointable if there is $T^*:E\rightarrow E$ such that $\langle T\xi,\eta\rangle=\langle \xi, T^*\eta\rangle $ for all $\xi,\eta \in E$. The set of all adjointable operators on $E$ is a $C^*$-algebra denoted $\mathcal{L}(E)$.
		\item[(c)] Given $\xi,\eta \in E$, the operator $\theta_{\xi,\eta}$ on $E$ is defined by $\theta_{\xi,\eta}\zeta=\xi \langle \eta,\zeta\rangle$. The algebra of the (generalized) compact operators $K(E)$ is $\overline{\{\theta_{\xi,\eta}: \xi,\eta \in E\}}$ and it is a sub $C^*$-algebra (in fact, an ideal) of $\mathcal{L}(E)$.
		\item[(d)] A $C^*$-correspondence $E$ over a $C^*$-algebra $A$ is a right Hilbert $C^*$-module over $A$ that is also endowed with a $^*$-homomorphism $\varphi:A \rightarrow \mathcal{L}(E)$ that can be viewed as left multiplication of $A$ on $E$, making $E$ a bimodule over $A$. 
	\end{enumerate}
\end{defn}

In this paper, the algebra  will be unital and $\varphi$ will be a unital map.

\begin{lem}
\begin{enumerate}
\item[(1)] $F$ is a $C^*$-correspondence over the $C^*$-algebra $\mathcal{D}$ where the left and right actions are defined by multiplication and the $\mathcal{D}$-valued inner product is $\langle T,S\rangle=T^*S$.
\item[(2)] $\mathcal{D}=\{T\in \mathcal{T}: TP_k=P_kT , \;k\geq 0 \}\subseteq \Sigma^{\oplus}_{k\geq 0} P_k\mathcal{T}P_k$ and $F=\{T\in \mathcal{T}: TP_{k}=P_{k+1}T , \;k\geq 0 \}\subseteq \Sigma^{\oplus}_{k\geq 0} P_{k+1}\mathcal{T}P_k$.
\item[(3)] Writing $\varphi_F$ for the left action of $\mathcal{D}$ on $F$, we get $\ker(\varphi_F)=\mathcal{D}P_0=P_0\mathcal{T}P_0$.
\item[(4)] Write $K(F)$ for the algebra of generalized compact operators on $F$ then $K(F)=\varphi_F(\mathcal{D})=\mathcal{L}(F)$
\end{enumerate}
\end{lem}
\begin{proof}
The proof of the parts (1)-(3) is straightforward and is omitted.
For part (4), recall that $K(F)\subseteq \mathcal{L}(F)$ is the norm closed ideal generated by the operators $\theta_{\xi_1,\xi_2}$ (for $\xi_1,\xi_2 \in F$) where $$\theta_{\xi_1,\xi_2}\xi_3=\xi_1 \langle \xi_2,\xi_3\rangle=\xi_1\xi_2^*\xi_3.$$
But, then, $\theta_{\xi_1,\xi_2}=\varphi_F(\xi_1\xi_2^*)\in \varphi_F(\mathcal{D})$. Thus, $K(F)\subseteq \varphi_F(\mathcal{D})$ and
$$I=\varphi_F(I)=\varphi_F(P_0^{\perp})=\varphi_F(\Sigma_i S_iS_i^*)=\Sigma_i \theta_{S_i,S_i}\in K(F).$$
\end{proof}

Write $\mathcal{K}$ for the algebra of compact operators on the Fock space $\mathcal{F}(\mathbb{C}^d)$ and let $q:B(\mathcal{F}(\mathbb{C}^d))\rightarrow  B(\mathcal{F}(\mathbb{C}^d))/\mathcal{K}$ be the quotient map. It is straightforward to check that $q(F)$ is a $C^*$-correspondence over $q(\mathcal{D})$ (with the obvious operations and inner product).
\begin{lem}\label{qF1}
\begin{enumerate}
\item[(1)] $\ker(\varphi_{q(F)})=\{0\}$
\item[(2)] For $\xi,\eta \in F$ $$\theta_{q(\xi),q(\eta)}=\varphi_{q(F)}(q(\xi\eta^*)).$$
\item[(3)] $K(q(F))=\call(q(F))=\varphi_{q(F)}(\cald) $.

\end{enumerate}

\end{lem}
\begin{proof}
Note first that, for $d\in \cald$,
\begin{equation}\label{d}
d=dP_0+\Sigma_i dS_iS_i^*
\end{equation}
Assume now that $q(d)\in \ker(\varphi_{q(F)})$. Then, $q(dS_i)=0$ (as $S_i\in F$) and, thus, $dS_i\in \calk$ and also $dS_iS_i^*\in \calk$. It follows from (\ref{d}) that $q(d)=0$. This proves (1).

For (2) we compute $\theta_{q(\xi),q(\eta)}q(\zeta)=q(\xi)q(\eta)^*q(\zeta)=q(\xi\eta^*)q(\zeta)=\varphi_{q(F)}(q(\xi)q(\eta)^*)q(\zeta)$.
This shows that $K(q(F))\subseteq \varphi_{q(F)}(\cald)$. For the converse, we have, for $d\in \cald$, $\varphi_{q(F)}(q(d))=\Sigma_i \varphi_{q(F)}(q((dS_i)S_i^*))=\Sigma_i \theta_{q(dS_i),q(S_i)} \in K(q(F))$.
\end{proof}

We now write $u_i:=q(S_i)$ and $z:=q(Z)$ and, for a word $\alpha$ of length $n$, $u_{\alpha}=u_{\alpha_1}\cdots u_{\alpha_n}$. The following properties of $\{u_{\alpha}\}$ will be used repeatedly.

\begin{lem}\label{alpha}
	\begin{enumerate}
		\item [(1)] For words $\alpha,\beta$ of the same length, $u_{\alpha}^*u_{\beta}=\delta_{\alpha,\beta}$
		\item[(2)] For every $n>0$, $\Sigma_{|\alpha|=n}u_{\alpha}u_{\alpha}^*=I$.
	\end{enumerate}
\end{lem}
\begin{proof}
	Part (1) follows from the fact that $S_i^*S_i=I$ and $S_i^*S_j=0$ if $i\neq j$. For part (2) note that $\Sigma_{|\alpha|=n}S_{\alpha}S_{\alpha}^*=I-\Sigma_{k=0}^{n-1}P_k$ and $q(P_k)=0$.
\end{proof}

\begin{lem}\label{qF}
	$$q(F)=\Sigma_i u_iq(\cald)$$ where $u_i=q(S_i)$.
	More generally, if we write $F^n$ for $span\{a_1\cdots a_n :\; a_i\in F\}$ then
	$$q(F^n)=\Sigma_{|\alpha|=n}u_{\alpha}q(\cald).$$
\end{lem}
\begin{proof} Fix $n\geq 1$.
We have $\Sigma_{|\alpha|=n} u_{\alpha}u_{\alpha}^*=I$ and, therefore, $q(F^n)=\Sigma_{|\alpha|=n} u_{\alpha}u_{\alpha}^*q(F^n)=\Sigma_{|\alpha|=n} u_{\alpha} (u_{\alpha}^*q(F^n))\subseteq \Sigma_{\alpha} u_{\alpha}q(\cald)$. The converse inclusion is obvious.	
	\end{proof}

The main result of this section is the following theorem.

\begin{thm}\label{isomorphism}
The $C^*$-algebra $\mathcal{T}(\mathbb{C}^d,Z)/\mathcal{K}$ is $*$-isomorphic to the Cuntz-Pimsner algebra $\mathcal{O}(q(F),q(\mathcal{D}))$.
\end{thm}

Recall that, given a $C^*$-correspondence $X$ over a $C^*$-algebra $A$ and a $C^*$-algebra $B$. A $^*$-representation $\pi:\mathcal{O}(X,A)\rightarrow B$ is given by a pair of maps:
a contractive map $T:X\rightarrow B$ and a $^*$-homomorphism $\sigma: A\rightarrow B$ such that, for $x,y\in X$ and $a,b\in A$,
\begin{enumerate}
\item[(1)] $T(\varphi_X(a)x b)=\sigma(a)T(x)\sigma(b)$
\item[(2)] $T(x)^*T(y)=\sigma(\langle x,y \rangle)$
\item[(3)] $\sigma^{(1)}(\varphi_X(a))=\sigma(a)$ for every $a\in J_X:=\varphi_X^{-1}(K(X))\cap (\ker(\varphi_X))^{\perp}$ where $\sigma^{(1)}:K(X)\rightarrow B$ is a $^*$-homomorphism defined by $\sigma^{(1)}(\theta_{x,y})=T(x)T(y)^*$.

\end{enumerate}

\begin{lem}\label{homomorphism}
With $q(F)$ and $q(\mathcal{D})$ as above, define
$$ \sigma: q(\mathcal{D}) \rightarrow \mathcal{T}/\mathcal{K}$$ by $\sigma(q(d))=q(d)$ and
$$T: q(F) \rightarrow \mathcal{T}/\mathcal{K}$$ by $T(q(\xi))=q(\xi)$.

Then $\sigma$ and $T$ satisfy conditions (1)-(3) above and, thus, define a $^*$-homomorphism $\pi$ of $\mathcal{O}(q(F),q(\mathcal{D}))$ onto $\mathcal{T}/\mathcal{K}$.

\end{lem}
\begin{proof}
It is clear that $\sigma$ is a well defined $^*$-homomorphism and $T$ is a contraction. It is also clear that $T$ is a bimodule map (over $\sigma$). For (2), we write
$$T(q(\xi))^*T(q(\eta))=q(\xi)^*q(\eta)=q(\xi^*\eta)=\sigma(\langle q(\xi),q(\eta)\rangle). $$
For (3), note first that it follows from Lemma~\ref{qF1} that $J_{q(F)}=q(\cald)$ and, for $d\in \cald$,
$$\varphi_{q(F)}(q(d))=\Sigma_i \theta_{q(dS_i),q(S_i)}.$$
Thus $\sigma^{(1)}(\varphi_{q(F)}(q(d)))=\Sigma_i T(q(dS_i))T(q(S_i))^*=\Sigma_i q(dS_i)q(S_i)^*=q(d)=\sigma(q(d))$ proving (3). Since $\calt$ is generated by $\{I,S_i,Z\}$ and $I, Z \in \cald$ and $S_i\in F$, $\calt / \calk$ is generated by the images of $\sigma$ and $T$. Thus, $\pi$ is onto $\calt /\calk$.

\end{proof}

\begin{proof} (of Theorem~\ref{isomorphism})
In view of Lemma~\ref{homomorphism} all we need to prove is that the $*$-homomorphism $\pi$ is injective but this follows from Theorem 6.4 of \cite{Ka}, since $\sigma$ is injective and $\pi$ admits a gauge action. Indeed, the injectivity of $\sigma$
is clear from his definition, and if $d\in \mathcal{D}$ and $\xi\in F$, then $\gamma_t(d)=d$ and $\gamma_t(S_i)=e^{it}S_i$. Hence, $\gamma_t(\sigma(q(d)))=\gamma_t(q(d))=q(d)=\sigma(q(d))$ and $\gamma_t(T(q(S_i)))=e^{it}T(q(S_i))$.
\end{proof}


\section{The algebra $q(\cald)$ and its ideals}

In this section we shall study the algebra $q(\cald)$ and its ideals. We shall pay special attention to those ideals that are \emph{invariant} in the following sense.

\begin{defn}\label{idinv}
An ideal $J\subseteq q(\mathcal{D})$ is said to be invariant if 
$$q(F)^*Jq(F)\subseteq J .$$ Using Lemma~\ref{qF}, this is equivalent to
$$u_i^*Ju_j \subseteq J$$ for all $i,j$.	
	
	\end{defn}

Note that the automorphism group $\gamma=\{\gamma_t\}$, on $B(\mathcal{F}(\mathbb{C}^d))$, as defined in the discussion following Lemma~\ref{T}, can be used to define bounded projections $\Phi_j$ on $B(\mathcal{F}(\mathbb{C}^d))$ by
$$\Phi_j(X)=\frac{1}{2\pi}\int_0^{2\pi}e^{-ijt}\gamma_t(X)dt .$$

 In particular, for $A\in \mathcal{T}$, $A\in \mathcal{D}$ if and only if $\Phi_0(A)=A$.

\begin{lem}\label{qD}
	$$q(\cald)=C^*(\{u_{\alpha}z^mu_{\beta}^*, u_{\alpha}^*z^mu_{\beta}:\;|\alpha|=|\beta|,\; m\geq 0\}).$$
\end{lem}
\begin{proof}
Fix $a\in q(\cald)$ and write $a=q(A)$ for some $A\in \cald$. Since $A\in  \calt$, it is a norm limit of (noncommutative) polynomials in $Z, Z^*, S_i, S_i^*$. Since $\Phi_0$ is a contractive projection, and $A=\Phi_0(A)$, we can assume, by applying $\Phi_0$, that each of these polynomials lie in $\cald$. Each such polynomial is a finite sum of monomials and each monomial lies in the range of some $\Phi_j$. By applying $\Phi_0$, we can assume that each monomial lies in $\cald$. Thus, $A$ is a norm limit of polynomials in $\{Z,Z^*,S_i, S_i^*\}$ where each monomial lies in $\cald$. It follows that $a$ is a norm limit of polynomials in $\{z,z^*,u_i, u_i^*\}$ where each monomial lies in $q(\mathcal{D})$. It is straightforward to check (using Lemma~\ref{alpha}) that each monomial in $q(\cald)$ is a product of elements in $\{u_{\alpha}z^mu_{\beta}^*, u_{\alpha}^*z^mu_{\beta}:\;|\alpha|=|\beta|\geq 0,\; m\geq 0\}$ or their adjoints.	
	
	\end{proof}

Recall that a unital $C^*$-algebra is said to be \emph{finite} if it has no proper isometries.

\begin{lem}\label{finite}
	The algebra $q(\mathcal{D})$ is finite.
\end{lem}
\begin{proof}
	Suppose $v=q(V)$ is an isometry in $q(\mathcal{D})$. Then $v^*v=I$ and, thus $V^*V=I+K$ for a compact operator $K$. Thus $\lim_{k\rightarrow \infty}\|P_k(V^*V-I)P_k\|=0$. Fix $N$ such that, for every $k\geq N$, $\|P_kV^*VP_k-P_k\|<1/2$. It follows that, for such $k$, $P_kV^*VP_k$ is invertible in $B(P_k\mathcal{F}(\mathbb{C}^d))\cong M_{d^k}(\mathbb{C})$. Write $C_k$ for its inverse and note that
	\begin{equation}\label{Cbdd} 
	\|C_k\|\leq \frac{1}{1-\|P_k-P_kV^*VP_k\|}<2.
	\end{equation}  
	Since $B(P_k\mathcal{F}(\mathbb{C}^d))$ is a matrix algebra and $P_kV^*VP_k=P_kV^*P_kVP_k$, it follows that $P_kVP_k$ is also invertible. Write $B_k$ for its inverse (so that $C_k=B_k^*B_k$) and $B:=I\oplus I \oplus \cdots \oplus I \oplus B_N\oplus B_{N+1} \oplus \cdots \in B(\mathcal{F}(\mathbb{C}^d))$. It follows from (\ref{Cbdd}) that $B$ is bounded (in fact $\|B\|\leq \sqrt{2}$) and it is the inverse of $V':=I\oplus I \oplus \cdots \oplus I \oplus P_NVP_N\oplus P_{N+1}VP_{N+1} \oplus \cdots$. Thus $V'$ is invertible in $B(\mathcal{F}(\mathbb{C}^d))$ and $q(V')$ is invertible in $B(\mathcal{F}(\mathbb{C}^d))/\mathcal{K}$. But $q(V')=q(V)=v$. Thus $v$ is invertible in $B(\mathcal{F}(\mathbb{C}^d))/\mathcal{K}$. Since $v\in q(\mathcal{D})$, it is also invertible there.	
	
\end{proof}

For the next proposition note that we can view $q(\mathcal{D})$ as a $C^*$-correspondence over itself with left and right actions given by multiplication and the inner product is $\langle d_1,d_2\rangle:=d_1^*d_2$.

\begin{prop}\label{nonperiodic}
	\begin{enumerate}
		\item[(1)] Suppose $d>1$, then there is no map $v:q(\cald)\rightarrow q(F^n)$ that is an isomorphism of correspondences (that is, a surjective bimodule map that preserves the inner product).
		\item[(2)] Suppose $d=1$ and $u=q(S)$. Then $z$ commutes with $u^n$ if and only if  there is a map $v:q(\cald)\rightarrow q(F^n)$ that is an isomorphism of correspondences .
	\end{enumerate} 
\end{prop}
\begin{proof}
	Assume first that $d>1$. Fix $n\geq 1$ and $v:q(\cald)\rightarrow q(F^n)$ that is an isomorphism of correspondences.
	Fix two different words $\alpha,\beta$ of length $n$. Then $u_{\alpha},u_{\beta}\in q(F^n)$ and, thus, there are $d_{\alpha},d_{\beta}\in q(\mathcal{D})$ such that $v(d_{\alpha})=u_{\alpha}$ and $v(d_{\beta})=u_{\beta}$. Since $v$ preserves inner products, we have $d_{\alpha}^*d_{\beta}=\langle d_{\alpha},d_{\beta}\rangle=\langle v(d_{\alpha}),v(d_{\beta})\rangle = \langle u_{\alpha},u_{\beta}\rangle=u_{\alpha}^*u_{\beta}=0$. Similarly $d_{\alpha}^*d_{\alpha}=d_{\beta}^*d_{\beta}=I$. Thus $d_{\alpha},d_{\beta}$ are two isometries  in $q(\cald)$ with orthogonal ranges. By Lemma~\ref{finite}, this is impossible.
	
	Now, assume $d=1$ and $v:q(\cald)\rightarrow q(F^n)$ is an isomorphism of correspondences. There is $w\in q(\cald)$ such that $v(w)=u^n$ and $b\in q(\cald)$ such that $v(z)=u^nb$ (as $q(F^n)=u^nq(\cald)$). Then $v(z)=v(w)b=v(wb)$. Thus $z=wb$. Now, $u^{* n}zu^n=\langle v(w),v(zw) \rangle=\langle w,zw \rangle =w^*zw=w^*wz$ (since $q(\cald)$ is commutative when $d=1$). Similarly, $I=u^{* n}u^n=w^*w$ and we get $u^{* n}zu^n=z$. Thus, $z$ commutes with $u^n$.
	For the other direction, assume that $z$ commutes with $u^n$. By Lemma~\ref{qD} the algebra $q(\mathcal{D})$ is generated by $\{u^kz^mu^{* k}, u^{* k}z^mu^k \}$. But, writing $k=ln+r=(l+1)n-(n-r)$ (with $0 \leq r <n$), we have $u^{* k}z^mu^k=u^{n-r}u^{* (l+1)n}z^mu^{(l+1)n}u^{* (n-r)}=u^{n-r}z^mu^{* (n-r)}$.Thus, $q(\mathcal{D})$ is generated by $\{u^kz^mu^{* k}\}$. It follows that $u^n$ commutes with every $a\in q(\mathcal{D})$. Now we define $v:q(\mathcal{D})\rightarrow q(F^n)$ by $v(a)=u^na=au^n$ (so that $v(u^kz^mu^{* k})=u^nu^kz^mu^{* k}$). It is straightforward to check that $v$ is an isomorphism of correspondences.
	
	
\end{proof}

For a fixed $p\in \mathbb{N}$ we shall consider the following condition on $Z$:
\vspace{5mm}

\textbf{Condition A(p):} $\lim_{k\rightarrow \infty} \|I_p\otimes Z_k-Z_{k+p}\|=0$.

\vspace{5mm}
In particular, Condition A(p) holds if there is some $N>0$ such that, for $k\geq N$, $Z_{k+p}=I_{p}\otimes Z_k$, where $I_p$ is the identity matrix of size $d^p$.

Note also that, if Condition A(p) holds then it holds also for $\{Z_k^*\}$. That is, we have $\lim_{k\rightarrow \infty} \|I_p\otimes Z_k^*-Z_{k+p}^*\|=0$.

Recall that $q:\calt \rightarrow \calt/\calk$ is the quotient map and we write $u_i=q(S_i)$ and $z=q(Z)$.

\begin{lem}\label{uz}
$Z$ satisfies Condition A(p) if and only if $u_{\alpha}z=zu_{\alpha}$ for every word $\alpha$ with $|\alpha|=p$. In this case, we also have $u_{\alpha}^*z=zu_{\alpha}^*$ for $|\alpha|=p$.
\end{lem}
\begin{proof}
$u_{\alpha}z=zu_{\alpha}$ if and only if $S_{\alpha}Z-ZS_{\alpha}\in \calk$. This hold if and only if $S_{\alpha}ZS_{\alpha}^*-ZS_{\alpha}S_{\alpha}^*\in \calk$ (since $S_{\alpha}^*S_{\alpha}=I$). Thus, $u_{\alpha}z=zu_{\alpha}$ for every $|\alpha|=p$ if and only if $\Sigma_{|\alpha|=p} S_{\alpha}ZS_{\alpha}^*-\Sigma_{|\alpha|=p}ZS_{\alpha}S_{\alpha}^*\in \calk$. Since $I-\Sigma_{|\alpha|=p} S_{\alpha}S_{\alpha}^*=\Sigma_{i=0}^{p-1}P_i\in \calk$, this is satisfied if and only if $\Sigma_{|\alpha|=p} S_{\alpha}ZS_{\alpha}^*-Z\in \calk$. The last inclusion is equivalent to $||P_{k+p}(\Sigma_{|\alpha|=p} S_{\alpha}ZS_{\alpha}^*)P_{k+p}-P_{k+p}ZP_{k+p}||\rightarrow 0$ for $k\rightarrow \infty$. 
But $P_{k+p}ZP_{k+p}=Z_{k+p}$ and it is easy to check that $P_{k+p}(\Sigma_{|\alpha|=p} S_{\alpha}ZS_{\alpha}^*)P_{k+p}=I_p\otimes Z_k$ (simply apply it to $e_{j_1}\otimes \cdots \otimes e_{j_{k+p}}$). So this completes the proof.

The last statement follows from the fact that Condition A(p) also holds for $Z^*$ in place of $Z$.

\end{proof}

\begin{lem}\label{comm}
	Assuming Condition A(p), every element in $C^*(z)$ commutes with each $u_{\alpha}u_{\beta}^*$ for words $\alpha,\beta$ with $|\alpha|=|\beta|$.
\end{lem}
\begin{proof} Suppose $|\alpha|=|\beta|=m$ and fix $n$ such that $m\leq np$.
	Since $\Sigma_{|\gamma|=np-m}u_{\gamma}u_{\gamma}^*=I$ and  $z$ is commutes with $u_{\delta}$ for $|\delta|=np$, we have
	$$u_{\alpha}u_{\beta}^*z=\Sigma_{|\gamma|=np-m} u_{\alpha}u_{\gamma}u_{\gamma}^*u_{\beta}^*z=\Sigma_{|\gamma|=np-m} u_{\alpha}u_{\gamma}zu_{\gamma}^*u_{\beta}^*=\Sigma_{|\gamma|=np-m} zu_{\alpha}u_{\gamma}u_{\gamma}^*u_{\beta}^*=zu_{\alpha}u_{\beta}^*. $$
	
	The same argument also works for $z^*$ and it follows that $u_{\alpha}u_{\beta}^*$ commutes with every operator in $C^*(z)$.	
	
\end{proof}

\begin{lem}\label{qDAp}
	If, for some $p>0$, condition A(p) holds then 
	$$q(\cald)=C^*(\{u_{\alpha}z^mu_{\beta}^*:\;|\alpha|=|\beta|,\; m\geq 0\}).$$	
\end{lem}
\begin{proof}
	We need to show that, for $|\alpha|=|\beta|$ and $m\geq 0$, $u_{\alpha}^*z^mu_{\beta}\in C^*(\{u_{\alpha}z^mu_{\beta}^*:\;|\alpha|=|\beta|,\; m\geq 0\})$. 
	First, write $\alpha=\alpha'\alpha''$ and $\beta=\beta'\beta''$ where $|\alpha'|=|\beta'|$ is a multiple of $p$ and $0\leq |\alpha''|=|\beta''|=k<p$. Then $u_{\alpha}^*z^mu_{\beta}=u_{\alpha''}^*u_{\alpha'}^*z^mu_{\beta'}u_{\beta''}=u_{\alpha''}^*z^mu_{\alpha'}^*u_{\beta'}u_{\beta''}$ which is $0$ if $\alpha'$ is different from $\beta'$ and $u_{\alpha''}^*z^mu_{\beta''}$ otherwise. Thus, we can assume that $|\alpha|=|\beta|=k<p$.
		Then, we write
		$$u_{\alpha}^*z^mu_{\beta}=\Sigma_{|\gamma|=p-k} u_{\gamma}u_{\gamma}^*u_{\alpha}^*z^mu_{\beta}=
\Sigma_{|\gamma|=p-k} u_{\gamma}(u_{\alpha}u_{\gamma})^*z^mu_{\beta}=\Sigma_{|\gamma|=p-k} u_{\gamma}z^m(u_{\alpha}u_{\gamma})^*u_{\beta}$$ which is $0$ if $\alpha$ is different from $\beta$ and $\Sigma u_{\gamma}z^mu_{\gamma}^*$ otherwise. In any case, we see that $u_{\alpha}^*z^mu_{\beta}$ lies in 
$C^*(\{u_{\alpha}z^mu_{\beta}^*:\;|\alpha|=|\beta|,\; m\geq 0\})$.
\end{proof}

We now write for $n\geq 0$,
\begin{equation}\label{Cn}
\mathcal{C}_n=C^*(\{u_{\alpha}z^lu_{\beta}^* :\; l\geq 0,\; np\leq |\alpha|=|\beta|<(n+1)p\}).
\end{equation} 
In particular,
 \begin{equation}\label{C0}
 \mathcal{C}_0=C^*(\{u_{\alpha}z^lu_{\beta}^* :\; l\geq 0,\; 0\leq |\alpha|=|\beta|<p\}).
 \end{equation}
\begin{lem}\label{maps}
	\begin{enumerate} Assuming Condition A(p) for some $p>0$, we have the following.
		\item[(1)] For every $n\geq 0$, $\mathcal{C}_n\subseteq \mathcal{C}_{n+1}$ . If $\iota_n$ is the inclusion map then $$\iota_n(u_{\alpha}z^lu_{\beta}^*)=\Sigma_{|\gamma|=p} u_{\alpha}u_{\gamma}z^lu_{\gamma}^*u_{\beta}^*.$$
		\item[(2)] For every $n\geq 0$, $M_{d^p}(\mathbb{C})\otimes \mathcal{C}_n$ is isomorphic to $\mathcal{C}_{n+1}$ via the map $g_n:M_{d^p}(\mathbb{C})\otimes \mathcal{C}_n\rightarrow \mathcal{C}_{n+1}$ defined by
		
		 $$g_n(e_{\gamma,\zeta}\otimes a)= u_{\gamma}au_{\zeta}^* $$ 
		 
		 where $a\in \mathcal{C}_n$ and $e_{\gamma,\zeta}$ is a matrix unit of $M_{d^p}(\mathbb{C})$ (where $|\gamma|=|\zeta|=p$).
\item[(3)] For every $n\geq 0$, $\mathcal{C}_n$ is isomorphic to $M_{d^{pn}}(\mathbb{C})\otimes \mathcal{C}_0$. The isomorphism is given by the map $\phi_n$ where
$$\phi_n(e_{\alpha,\beta}\otimes a)=u_{\alpha}au_{\beta}^*$$ for $a\in \mathcal{C}_0$ and $|\alpha|=|\beta|=np$. In fact,
$$\phi_n=g_{n-1}(I_{M_{d^p}}\otimes g_{n-2})\cdots (I_{M_{d^{p(n-1)}}}\otimes g_0).$$ 
\item[(4)] Write $\iota_n$ for the inclusion of $\mathcal{C}_n$ into $\mathcal{C}_{n+1}$ and $$\psi_n:=\phi_{n+1}^{-1}\circ \iota_n \circ \phi_n:M_{d^{pn}}(\mathbb{C})\otimes \mathcal{C}_0 \rightarrow M_{d^{p(n+1)}}(\mathbb{C})\otimes \mathcal{C}_0.$$
	Then $$\psi_n(e_{\alpha,\beta}\otimes u_{\tau}z^lu_{\sigma}^*)=\Sigma_{|\rho|=p-|\tau|} e_{\alpha \tau\rho,\beta \sigma\rho}\otimes \Sigma_{|\delta|=|\tau|}u_{\delta}z^l u_{\delta}^* .$$
\item[(5)] For every $n\geq 0$ and $1\leq i,j\leq d$,
$$u_i^*\mathcal{C}_nu_j\subseteq \mathcal{C}_n.$$		
	\end{enumerate}
\end{lem}
\begin{proof}
To prove (1), take $u_{\alpha}z^lu_{\beta}^*\in \mathcal{C}_n$ (with $np\leq |\alpha|=|\beta|	<(n+1)p$) and compute
$$u_{\alpha}z^lu_{\beta}^*=\Sigma_{|\gamma|=p} u_{\alpha}u_{\gamma}u_{\gamma}^*z^lu_{\beta}^*=\Sigma_{|\gamma|=p} u_{\alpha}u_{\gamma}z^lu_{\gamma}^*u_{\beta}^*\in \mathcal{C}_{n+1}$$ since $(n+1)p\leq|\alpha \gamma|=|\beta \gamma|< (n+2)p$.

To prove (2), define $g_n$ as the linear map defined by
 $$g_n(e_{\gamma,\zeta}\otimes a)= u_{\gamma}au_{\zeta}^*\;\;\;,a\in \mathcal{C}_n,\; |\gamma|=|\zeta|=p. $$  
 It is straightforward to check  that it's range is contained in $\mathcal{C}_{n+1}$. It is also clear that this map is self adjoint. To prove that it is multiplicative, compute $(e_{\gamma_1,\zeta_1}\otimes a)(e_{\gamma_2,\zeta_2}\otimes b)=\delta_{\zeta_1,\gamma_2} e_{\gamma_1,\zeta_2}\otimes ab$ and $g_n(e_{\gamma_1,\zeta_1}\otimes a)g_n(e_{\gamma_2,\zeta_2}\otimes b)
=u_{\gamma_1}au_{\zeta_1}^*u_{\gamma_2}bu_{\zeta_2}^*=\delta_{\zeta_1,\gamma_2}u_{\gamma_1}abu_{\zeta_2}^*=g_n(\delta_{\zeta_1,\gamma_2} e_{\gamma_1,\zeta_2}\otimes ab)$. To prove surjectivity, fix $u_{\alpha}z^lu_{\beta}^*\in \mathcal{C}_{n+1}$ (with $(n+1)p\leq |\alpha|=|\beta|<(n+2)p$). 
Write $\alpha=\alpha'\alpha''$ and $\beta=\beta'\beta''$ with $|\alpha'|=|\beta'|=p$ and $|\alpha''|=|\beta''|=|\alpha|-p$. Then $u_{\alpha}z^lu_{\beta}^*=u_{\alpha'}(u_{\alpha''}z^lu_{\beta''}^*)u_{\beta'}^*=g_n(e_{\alpha',\beta'}\otimes u_{\alpha''}z^lu_{\beta''}^*)$. Thus $g_n$ is surjective.

To prove injectivity, assume that $0=g_n(\Sigma e_{\alpha,\beta}\otimes a_{\alpha,\beta})=\Sigma u_{\alpha}a_{\alpha,\beta}u_{\beta}^*$. Thus $$\Sigma u_{\alpha}a_{\alpha,\beta}u_{\beta}^*=0.$$ Fix $\alpha_0$ and $\beta_0$ with $|\alpha_0|=|\beta_0|=p$ and multiply this equation on the left by $u_{\alpha_0}^*$ and, on the right, by $u_{\beta_0}$. We get $a_{\alpha_0,\beta_0}=0$ for every $\alpha_0$ and $\beta_0$.

This proves (2) and (3) follows from this.

To prove (4) we have to trace the definitions of the maps involved. From the proof of (1) we get
$$\iota_n(u_{\alpha}z^lu_{\beta}^*)=\Sigma_{|\gamma|=p} u_{\alpha}u_{\gamma}z^lu_{\gamma}^*u_{\beta}^*.$$
Thus $$\psi_n(e_{\alpha,\beta}\otimes u_{\tau}z^lu_{\sigma}^*)=\phi_{n+1}^{-1}(\iota_n(u_{\alpha}u_{\tau}z^lu_{\sigma}^*u_{\beta}^*))=\phi_{n+1}^{-1}(\Sigma_{|\gamma|=p}u_{\alpha}u_{\tau}u_{\gamma}z^lu_{\gamma}^*u_{\sigma}^*u_{\beta}^*).
$$ Now, for every $\gamma$ with $|\gamma|=p$ we write $\gamma=\rho \delta$ with $|\rho|=p-|\tau|$ and $|\delta|=|\tau|$. Then we get 
$$\psi_n(e_{\alpha,\beta}\otimes u_{\tau}z^lu_{\sigma}^*)=\phi_{n+1}^{-1}(\Sigma_{|\rho|=p-|\tau|, |\delta|=|\tau|}u_{\alpha}u_{\tau}u_{\rho}u_{\delta}z^lu_{\delta}^*u_{\rho}^*u_{\sigma}^*u_{\beta}^*)$$  $$=\Sigma_{|\rho|=p-|\tau|}e_{\alpha \tau \rho,\beta \sigma \rho}\otimes \Sigma_{|\delta|=|\tau|}u_{\delta}z^lu_{\delta}^* .$$

For (5), it is enough to take $a\in \mathcal{C}_n$ to be a generator $u_{\alpha}z^lu_{\beta}^*$ with $np\leq |\alpha|=|\beta|< (n+1)p$. Assume first that $|\alpha|>np$. Then we can write $\alpha=k\alpha'$ and $\beta=l\beta'$ for some $k,l$ and $np\leq |\alpha'|=|\beta'|< (n+1)p$, Then $u_i^*u_{\alpha}z^lu_{\beta}^*u_j=\delta_{i,k}\delta_{j,l}u_{\alpha'}z^lu_{\beta'}^*\in \mathcal{C}_n$. Now assume that $|\alpha|=|\beta|=np$ and $n\neq 0$. Then, a similar argument shows that $u_i^*u_{\alpha}z^lu_{\beta}^*u_j=\delta_{i,k}\delta_{j,l}u_{\alpha'}z^lu_{\beta'}^*$ and $|\alpha'|=|\beta'|=np-1<np$. Thus, in this case $u_i^*u_{\alpha}z^lu_{\beta}^*u_j\in \mathcal{C}_{n-1}\subseteq \mathcal{C}_n$ (where we used part (1)). To complete the proof we have to show that $u_i^*z^lu_j\in \mathcal{C}_0$. For that, we write $u_i^*z^lu_j=\Sigma_{|\gamma|=p-1}u_{\gamma}u_{\gamma}^*u_i^*z^lu_j=\Sigma_{|\gamma|=p-1}u_{\gamma}z^lu_{\gamma}^*u_i^*u_j=\delta_{i,j}\Sigma_{|\gamma|=p-1}u_{\gamma}z^lu_{\gamma}^*\in \mathcal{C}_0$.	
	
\end{proof}

Since $q(\mathcal{D})=\overline{\cup \mathcal{C}_n}=\lim_{n \rightarrow \infty} (\mathcal{C}_n,\iota_n)$, we have

\begin{cor}\label{dirlim}
	 Assuming Condition A(p),
$$q(\mathcal{D})\cong \lim_{n \rightarrow \infty} (M_{d^{np}}\otimes \mathcal{C}_0, \psi_n) .$$	
\end{cor}

Since the case  $p=1$ (so that Condition A(1) holds) is simpler, we first take care of this case.

\textbf{Case I : $p=1$}

If $p=1$, $\mathcal{C}_0=C^*(z)$ and  $$\psi_n:M_{d^{n}}(\mathbb{C})\otimes \mathcal{C}_0 \rightarrow M_{d^{(n+1)}}(\mathbb{C})\otimes \mathcal{C}_0$$ is defined by
 $$\psi_n(e_{\alpha,\beta}\otimes z^l)=\Sigma_{i} e_{\alpha i,\beta i}\otimes z^l  .$$
 
 Thus, we have the following.
 
 \begin{cor}\label{A1}
 	 Assuming Condition A(1),
 	$$q(\mathcal{D})\cong M_{d^{\infty}}\otimes C^*(z) $$
 	where $M_{d^{\infty}}$ is the UHF algebra obtained as a direct limit of the algebras $M_{d^n}(\mathbb{C})$.
 \end{cor}

It follows that, using this isomorphisms, we can write every closed ideal in $q(\cald)$ as
$M_{d^{\infty}}\otimes J_0$ for some closed ideal $J_0\subseteq C^*(z)$.


\begin{lem}\label{simplicityA1}
	 Assuming Condition A(1), every ideal $J$ in $q(\mathcal{D})$ is invariant in the sense of Definition~\ref{idinv}. That is,
	for every ideal $J\subseteq q(\cald)$ and every $1\leq i,j\leq d$, $u_i^*Ju_j\subseteq J$. In fact, $$q(F)^*Jq(F)\subseteq J.$$
	Thus, in this case, $q(\mathcal{D})$ has no \emph{non trivial} invariant ideals if and only if it has no non trivial ideals and that is the case if and only if $C^*(z)$ has no non trivial ideals. 
	
	If $z$ is a normal operator, this happens if and only if 
	$z\in \mathbb{C}I$ which is the case if and only if $Z_k\rightarrow \lambda I$ for some $\lambda \in \mathbb{C}$ (and $\lambda \neq 0$ by our standing assumptions on $\{Z_k\}$).
\end{lem}
\begin{proof} Let $J$ be an ideal in $q(\cald)$.
We shall show that, for every $i,j$, $u_i^*Ju_j\subseteq J$. So, we fix $J$ and $i,j$. There is an ideal $J_0\subseteq C^*(z)$
such that $J$ is generated by elements of the form $a=u_{\alpha}u_{\beta}^*c$ where $|\alpha|=|\beta|$ and $c\in C^*(z)$. 
If $|\alpha|=|\beta|=0$ then $a=c$ and $u_i^*au_j=u_i^*cu_j=u_i^*u_jc=\delta_{i,j}c\in J$.
Suppose now that $0<|\alpha|=|\beta|$ and 
  consider $u_i^*(u_{\alpha}u_{\beta}^*c)u_j$. This would be different from $0$ only if $\alpha=i \alpha_0$ and $\beta=j \beta_0$ for some words $\alpha_0, \beta_0$. In this case,
$$u_i^*(u_{\alpha}u_{\beta}^*c)u_j=u_{\alpha_0}u_{\beta_0}^*c \in J$$
completing the proof of invariance of $J$. For the last statement of the lemma it is clear that, if $z$ is a normal operator and $C^*(z)$ has no closed non trivial ideals then $sp(z)$ should be a single point.

\end{proof}

\textbf{Case II: $p>1$}

For $p>1$, we cannot write $\psi_n=t_n\otimes id_{\mathcal{C}_0}$ where $t_n:M_{d^{pn}}\rightarrow M_{d^{p(n+1)}}$ and, thus, we cannot conclude that $q(\mathcal{D})\cong M_{d^{\infty}}\otimes \mathcal{C}_0$. We will have to look closer at the structure of $\mathcal{C}_0$.

\begin{defn}
	A family $\{t_{\alpha,\beta}\}_{|\alpha|=|\beta|=n}$ in a unital $C^*$-algebra is called a family of $M_{d^n}$-matrix units if, for every $|\alpha|=|\beta|=|\sigma|=|\tau|=n$, $t_{\alpha,\beta}t_{\sigma,\tau}=\delta_{\beta,\sigma} t_{\alpha,\tau}$, $t_{\alpha,\beta}^*=t_{\beta,\alpha}$ and $\Sigma_{|\gamma|=n}t_{\gamma,\gamma}=I$.
\end{defn}

Clearly, the family $\{u_{\alpha}u_{\beta}^*:\; |\alpha|=|\beta|=n\}$ is a family of $M_{d^n}$-matrix units. The proof of the following lemma is straightforward and is omitted.

We write $\xi_m$ for the word of length $m$, $\xi_m=11\ldots1$ ($m$ times).

\begin{lem}\label{M2C}
	If $\mathcal{A}$ is a unital $C^*$-algebra that contains a family of $M_{d^n}$-matrix units then $\mathcal{A}$ is isomorphic to $M_{d^n}\otimes (t_{\xi_n,\xi_n}\mathcal{A}t_{\xi_n,\xi_n})$ where the isomorphism
	$\tau:\mathcal{A}\rightarrow M_{d^n}\otimes(t_{\xi_n,\xi_n}\mathcal{A}t_{\xi_n,\xi_n})$ is given by 
	$\tau(a)=\Sigma_{|\alpha|=|\beta|=n} e_{\alpha,\beta}\otimes t_{\xi_n,\alpha}at_{\beta,\xi_n}$. Also, $\tau^{-1}(e_{\alpha,\beta}\otimes a)=t_{\alpha,\xi_n}at_{\xi_n,\beta}$.
\end{lem}

For simplicity, we write $q=p-1$.

In our notation here, we write $\mathcal{C}_{00}=u_{\xi_q}u_{\xi_q}^*\mathcal{C}_{0}u_{\xi_q}u_{\xi_q}^*$ and then, by the lemma (using $t_{\alpha,\beta}=u_{\alpha}u_{\beta}^*$), we get 

$$\mathcal{C}_0\cong M_{d^q}\otimes\mathcal{C}_{00}$$ 
where the isomorphism $\tau:\mathcal{C}_0\rightarrow M_{d^q}\otimes \mathcal{C}_{00}$ is

$$\tau(a)=\Sigma_{|\alpha|=|\beta|=q} e_{\alpha,\beta}\otimes u_{\xi_q}u_{\alpha}^*au_{\beta}u_{\xi_q}^*$$ for $a\in \mathcal{C}_0$ and
 
$$\tau^{-1}(e_{\alpha,\beta}\otimes b)=u_{\alpha}u_{\xi_q}^*bu_{\xi_q}u_{\beta}^*$$ for $b\in \mathcal{C}_{00}$.

Using (\ref{C0}), we get

\begin{lem}\label{C00}
	$$\mathcal{C}_{00}=C^*(\{u_{\xi_m}cu_{\xi_{q-m}}u_{\xi_q}^*:\; c\in C^*(z)\;0\leq m\leq q\;\})$$ (where $u_{\xi_0}$ is $I$).
\end{lem}
\begin{proof}
	$\mathcal{C}_0$ is generated by $u_{\alpha}cu_{\beta}^*$ for $0\leq |\alpha|=|\beta|\leq q$ and $c\in C^*(z)$. Fix such a generator and write $m=|\alpha|=|\beta|$. Using the definition of $\mathcal{C}_{00}$, we look at $u_{\xi_q}u_{\xi_q}^*u_{\alpha}cu_{\beta}^*u_{\xi_q}u_{\xi_q}^*$. Now, $u_{\xi_q}^*u_{\alpha}=u_{\xi_{q-m}}^*u_{\xi_m}^*u_{\alpha}$ and this is $0$ unless $\alpha=\xi_m$ and, in this case, it is equal to $u_{\xi_{q-m}}^*$. So we can replace $u_{\xi_q}u_{\xi_q}^*u_{\alpha}$ by $u_{\xi_q}u_{\xi_{q-m}}^*$. Similarly, $u_{\beta}^*u_{\xi_q}u_{\xi_q}^*$ can be replaced by $u_{\xi_{q-m}}u_{\xi_q}^*$. Thus we now consider $u_{\xi_q}u_{\xi_{q-m}}^*c u_{\xi_{q-m}}u_{\xi_q}^*$. Writing $u_{\xi_q} =u_{\xi_m}u_{\xi_{q-m}}$ and using Lemma ~\ref{comm} we compute $u_{\xi_q}u_{\xi_{q-m}}^*c u_{\xi_{q-m}}u_{\xi_q}^*=u_{\xi_m}u_{\xi_{q-m}}u_{\xi_{q-m}}^*c u_{\xi_{q-m}}u_{\xi_q}^*=u_{\xi_m}cu_{\xi_{q-m}}u_{\xi_{q-m}}^* u_{\xi_{q-m}}u_{\xi_q}^*=u_{\xi_m}cu_{\xi_{q-m}}u_{\xi_q}^*$, and we are done.
	\end{proof}

Computed on the generators of $\mathcal{C}_0$, we have, for $0\leq |\rho|=|\sigma|=m\leq q=p-1$ and $c\in C^*(z)$,
$\tau(u_{\rho}cu_{\sigma}^*)=\Sigma_{|\alpha|=|\beta|=q} e_{\alpha,\beta}\otimes u_{\xi_q}u_{\alpha}^*u_{\rho}cu_{\sigma}^*u_{\beta}u_{\xi_q}^*$. Now, $u_{\sigma}^*u_{\beta}=u_{\alpha}^*u_{\rho}=0$ unless $\alpha=\rho \alpha'$ and $\beta=\sigma \beta'$ (for some $\alpha',\beta'$ with $|\alpha'|=|\beta'|=q-m$) and, in this case, $u_{\sigma}^*u_{\beta}=u_{\beta'}$ and $u_{\alpha}^*u_{\rho}=u_{\alpha'}$. We also have $u_{\xi_q}=u_{\xi_m}u_{\xi_{q-m}}$. Thus $\tau(u_{\rho}cu_{\sigma}^*)=\Sigma_{|\alpha'|=|\beta'|=q-m}e_{\rho \alpha',\sigma \beta'}\otimes u_{\xi_m}(u_{\xi_{q-m}}u_{\alpha'}^*)cu_{\beta'}u_{\xi_q}^*$. But $(u_{\xi_{q-m}}u_{\alpha'}^*)cu_{\beta'}=c(u_{\xi_{q-m}}u_{\alpha'}^*)u_{\beta'}$ and $u_{\alpha'}^*u_{\beta'}=0$ unless $\alpha'=\beta'$ and $u_{\alpha'}^*u_{\alpha'}=I$. Thus, finally, we have
\begin{equation}\label{tau}
\tau(u_{\rho}cu_{\sigma}^*)=\Sigma_{|\alpha'|=q-m}e_{\rho \alpha',\sigma \alpha'}\otimes u_{\xi_m}cu_{\xi_{q-m}}u_{\xi_q}^*.
\end{equation}  
It will also be useful to compute what $\tau^{-1}$ does to the generators of $M_{d^q}\otimes \mathcal{C}_{00}$. When $c\in C^*(z)$ , $|\alpha|=|\beta|=q$ and $0\leq m \leq q$, we have $\tau^{-1}(e_{\alpha,\beta}\otimes u_{\xi_m}cu_{\xi_{q-m}}u_{\xi_q}^*)=u_{\alpha}u_{\xi_q}^* u_{\xi_m}cu_{\xi_{q-m}}u_{\xi_q}^*u_{\xi_q}u_{\beta}^*=u_{\alpha}u_{\xi_{q-m}}^*cu_{\xi_{q-m}}u_{\beta}^*$, since $u_{\xi_q}^*=u_{\xi_{q-m}}^*u_{\xi_m}^*$ and $u_{\xi_m}^*u_{\xi_m}=I$. Finally, since $u_{\xi_{q-m}}=u_{\xi_{m}}^*u_{\xi_{q}}$, we can write
\begin{equation}\label{tauinv}
\tau^{-1}(e_{\alpha,\beta}\otimes u_{\xi_m}cu_{\xi_{q-m}}u_{\xi_q}^*)=(u_{\alpha}u_{\xi_q}^*)(u_{\xi_m}cu_{\xi_m}^*)(u_{\xi_q}u_{\beta}^*).
\end{equation}


	

\begin{lem}\label{psintilde}
Let $\tau$ be as above and write
 $$\tilde{\tau}_n=id_{M_{d^{pn}}}\otimes \tau:M_{d^{pn}}\otimes \mathcal{C}_0 \rightarrow M_{d^{pn}}\otimes M_{d^{p-1}} \otimes \mathcal{C}_{00}\cong M_{d^{pn+p-1}}\otimes \mathcal{C}_{00}$$
  and $\tilde{\psi}_n=\tilde{\tau}_{n+1}\circ \psi_n \circ \tilde{\tau}_n^{-1}:M_{d^{pn+p-1}}\otimes \mathcal{C}_{00}\rightarrow    M_{d^{pn+2p-1}}\otimes \mathcal{C}_{00}$. Then, for $B\in M_{d^{pn}}$,  $e_{\gamma,\delta}\in M_{d^q}$ and $a\in \mathcal{C}_{00}$,
\begin{equation}\label{eqpsintilde}
\tilde{\psi}_n(B\otimes e_{\gamma,\delta}\otimes a)=B\otimes e_{\gamma,\delta} \otimes (\Sigma_{|\rho|=p}e_{ \rho, \rho})\otimes a.
\end{equation} 	
	
	\end{lem} 
\begin{proof}
	Recall that we write $q$ for $p-1$. It is enough to prove it for $B=e_{\alpha,\beta}$ (with $|\alpha|=|\beta|=pn$) and $a=u_{\xi_m}cu_{\xi_{q-m}}u_{\xi_q}^*$ where $0\leq m\leq q$ and $c\in C^*(z)$. Thus, we shall fix such $\alpha$ and $\beta$ and $c\in C^*(z)$ and prove
$$\tilde{\psi}_n(e_{\alpha,\beta}\otimes e_{\gamma,\delta}\otimes u_{\xi_m}cu_{\xi_{q-m}}u_{\xi_q}^* )=\Sigma_{|\gamma|=p}e_{\alpha \gamma \rho,\beta \delta \rho}\otimes u_{\xi_m}cu_{\xi_{q-m}}u_{\xi_q}^*.$$ 

Indeed, 	
$$\tilde{\psi}_n(e_{\alpha,\beta}\otimes e_{\gamma,\delta}\otimes u_{\xi_m}cu_{\xi_{q-m}}u_{\xi_q}^* )=\tilde{\tau}_{n+1}(\psi_n(e_{\alpha,\beta}\otimes (u_{\gamma}u_{\xi_q}^*)(u_{\xi_m}cu_{\xi_m}^*)(u_{\xi_q}u_{\delta}^*) ))$$  $$=\tilde{\tau}_{n+1}(\psi_n(e_{\alpha,\alpha}\otimes u_{\gamma}u_{\xi_q}^*)\psi_n(e_{\alpha,\beta}\otimes u_{\xi_m}cu_{\xi_m}^*)\psi_n(e_{\beta,\beta}\otimes u_{\xi_q}u_{\delta}^*))$$  $$=\tilde{\tau}_{n+1}((\Sigma_i e_{\alpha \gamma i,\alpha \xi_q i}\otimes I)(\Sigma_{|\rho|=p-m}e_{\alpha \xi_m \rho,\beta \xi_m \rho}\otimes \Sigma_{|\sigma|=m}u_{\sigma}cu_{\sigma}^*)(\Sigma_j e_{\beta \xi_q j,\beta \delta j}\otimes I))$$  $$=\tilde{\tau}_{n+1}(\Sigma_{i,j, |\rho|=p-m}e_{\alpha \gamma i, \alpha \xi_q i}e_{\alpha \xi_m \rho, \beta \xi_m \rho}e_{\beta \xi_q j, \beta \delta j}\otimes \Sigma_{|\sigma|=m}u_{\sigma}cu_{\sigma}^*)$$  $$=\tilde{\tau}_{n+1}(\Sigma_i e_{\alpha \gamma i,\beta \delta i}\otimes \Sigma_{|\sigma|=m}u_{\sigma}cu_{\sigma}^*) . $$ 
For the last equality we used the fact that, for the product of the matrix units to be non zero we need $\xi_q i=\xi_m\rho=\xi_q j$ and it follows that $i=j$ and $\rho=\xi_{q-m} i$.
Now,
$$\tilde{\tau}_{n+1}(\Sigma_i e_{\alpha \gamma i,\beta \delta i}\otimes \Sigma_{|\sigma|=m}u_{\sigma}cu_{\sigma}^*)=\Sigma_i e_{\alpha \gamma i,\beta \delta i}\otimes \tau(\Sigma_{|\sigma|=m}u_{\sigma}cu_{\sigma}^*).$$
We have $\tau(u_{\sigma}cu_{\sigma}^*)=\Sigma_{|\alpha'|=q-m}e_{\sigma \alpha',\sigma \alpha'}\otimes u_{\xi_m}cu_{\xi_{q-m}}u_{\xi_q}^*$ and, thus, we get  
$\Sigma_i e_{\alpha \gamma i,\beta \delta i}\otimes (\Sigma_{|\sigma|=m} \Sigma_{|\alpha'|=q-m}e_{\sigma \alpha',\sigma \alpha'}\otimes u_{\xi_m}cu_{\xi_{q-m}}u_{\xi_q}^*)=\Sigma_{i, |\sigma'|=q}e_{\alpha \gamma i \sigma',\beta \delta i \sigma'}\otimes  u_{\xi_m}cu_{\xi_{q-m}}u_{\xi_q}^* $. Write $a= u_{\xi_m}cu_{\xi_{q-m}}u_{\xi_q}^* $ and $i\sigma'=\rho$ and note that then $e_{\alpha \gamma i \sigma',\beta \delta i \sigma'}\otimes a=e_{\alpha,\beta}\otimes e_{\gamma,\delta}\otimes e_{\rho,\rho}\otimes a$ to get the desired result.

	\end{proof}
\begin{cor}\label{UHFC00}
	Assume that Condition A(p) holds. Then
	$$q(\mathcal{D})\cong M_{d^{\infty}}\otimes \mathcal{C}_{00}. $$
\end{cor}

\section{Examples}\label{examples}
In this section we assume that $d=2$ and Condition A(2) holds so that $q(\mathcal{D})\cong M_{2^{\infty}}\otimes \mathcal{C}_{00}$ where $\mathcal{C}_{00}=C^*(\{c_1u_1u_1^*, u_1c_2u_1^*: \; c_1,c_2\in C^*(z)\;\})$. Note that it is a $C^*$-algebra with unit $u_1u_1^*$.

It will be convenient to write $I_l$ for the identity matrix of size $2^l\times 2^l$. (Recall that $Z_k$ is a $2^k\times 2^k$ matrix, as $d=2$ here).

	Fix $N\in \mathbb{N}$ and $A,B\in M_{2^N}(\mathbb{C})$ such that $A$ and $B$ are positive and invertible. Define $Z=\{Z_k\}$ by setting $Z_k=I$ for $k<N$, $Z_N=A$, $Z_{N+1}=B\oplus B=I_1\otimes B\in M_{2^(N+1)}(\mathbb{C})$ and, for $m>2$, $Z_{N+m}=I_2\otimes Z_{N+m-2}$. Then, Condition A(2) is clearly satisfied. 
	
\begin{lem}\label{AB} With the set up as above we get, for every $k\geq N+1$ (and $l=k-N-1$),
	\begin{enumerate}
		\item [(i)] If $l$ is even, $$P_kS_1S_1^*ZP_k=e_{1,1}\otimes I_l \otimes B $$ and $$P_k S_1 Z S_1^*P_k=e_{1,1}\otimes I_l \otimes A.$$
		\item[(ii)]  If $l$ is odd, $$P_kS_1S_1^*ZP_k=e_{1,1}\otimes I_l \otimes A $$ and $$P_k S_1 Z S_1^*P_k=e_{1,1}\otimes I_l \otimes B.$$
		\item[(iii)] For every $k\geq N+1$,
		$$P_kS_1S_1^*P_k=e_{1,1}\otimes I_l\otimes I_N.$$
	\end{enumerate}
It follows that, for every (non commutative) polynomial $p(x,y,z)$, the following are equivalent
\begin{enumerate}
	\item [(1)]  $p(A,B,I)=p(B,A,I)=0$.
	\item[(2)]  $P_kp(S_1S_1^*Z,S_1ZS_1^*,S_1S_1^*)P_k=0$ for every $k\geq N+1$.
	\item[(3)]  $P_kp(S_1ZS_1^*,S_1S_1^*Z,S_1S_1^*)P_k=0$ for every $k\geq N+1$.
	\item[(4)] $p(u_1u_1^*z,u_1zu_1^*,u_1u_1^*)=0$.
	\item[(5)] $p(u_1zu_1^*,u_1u_1^*z,u_1u_1^*)=0$.
\end{enumerate}  
\end{lem}
\begin{proof}
We will first prove (i). The proof of (ii) and (iii) is similar and is omitted. To prove (i), we assume that $l$ is even (so that $k-N$ is odd) and we apply the right hand side to $e_{n_1}\otimes \cdots \otimes e_{n_k}$.
$$P_kS_1S_1^*ZP_k(e_{n_1}\otimes \cdots \otimes e_{n_k})=P_kS_1S_1^*(I_{l+1}\otimes B)(e_{n_1}\otimes \cdots \otimes e_{n_k})$$  $$=P_kS_1S_1^*(e_{n_1}\otimes \cdots \otimes e_{n_{l+1}}\otimes B(e_{n_{l+2}}\otimes \cdots \otimes e_{n_k}))=\delta_{n_1,1} e_1\otimes e_{n_2} \otimes \cdots \otimes e_{n_{l+1}}\otimes  B(e_{n_{l+2}}\otimes \cdots \otimes e_{n_k})$$  $$=(e_{1,1}\otimes I_l \otimes B)(e_{n_1}\otimes \cdots \otimes e_{n_k}).$$
Also,
$$P_k S_1 Z S_1^*P_k(e_{n_1}\otimes \cdots \otimes e_{n_k})=P_kS_1ZS_1^*(e_{n_1}\otimes \cdots \otimes e_{n_k})=\delta_{n_1,1} P_kS_1Z(e_{n_2}\otimes \cdots \otimes e_{n_k})$$  $$=	\delta_{n_1,1} P_kS_1(I_l\otimes A)(e_{n_2}\otimes \cdots \otimes e_{n_k})=	\delta_{n_1,1}P_kS_1(e_{n_2}\otimes\cdots \otimes e_{n_{l+1}}\otimes A(e_{n_{l+2}}\otimes \cdots \otimes e_{n_k}))$$  $$=\delta_{n_1,1}e_1\otimes e_{n_2}\otimes\cdots \otimes e_{n_{l+1}}\otimes A(e_{n_{l+2}}\otimes \cdots \otimes e_{n_k})=(e_{1,1}\otimes I_l \otimes A)(e_{n_1}\otimes \cdots \otimes e_{n_k}).$$

We now turn to prove the equivalence of (1)-(5).

Assume (1) holds. Then, for every $k=N+1+2m$ (that is, $l=2m$ is even), $P_kp(S_1S_1^*Z,S_1ZS_1^*,S_1S_1^*)P_k=e_{1,1}\otimes I_{2m} \otimes p(B,A,I) =0$ (using (i) and the assumption that (1) holds). Similarly, for $k=N+2m$ (that is, $l=2m-1$ is odd), we use (ii) to get $P_kp(S_1S_1^*Z,S_1ZS_1^*,S_1S_1^*)P_k=e_{1,1}\otimes I_{2m-1} \otimes p(A,B,I) =0$. Thus, (1) implies (2). The arguments above can be reversed to show that, in fact, (1) and (2) are equivalent.

The proof that (1) and (3) are equivalent is almost identical and is omitted.

Since $p(u_1u_1^*z,u_1zu_1^*,u_1u_1^*)=q(p(S_1S_1^*Z,S_1ZS_1^*,S_1S_1^*))$, (4) is saying that $p(S_1S_1^*Z,S_1ZS_1^*,S_1S_1^*)$ is compact and this is equivalent to
\begin{equation}\label{cpt}
\lim_k \|P_k(p(S_1S_1^*Z,S_1ZS_1^*,S_1S_1^*))P_k\|=0.
\end{equation}
It is now clear that (2) implies (4) (and, therefore, (1) implies (4)). To show that (4) implies (1), assume that (\ref{cpt}) holds. But then, by considering a subsequence, $$\lim_m \|P_{N+1+2m}(p(S_1S_1^*Z,S_1ZS_1^*,S_1S_1^*))P_{N+1+2m}\|=0.$$ But this implies that $p(B,A,I)=0$. Similarly, for the subsequence $\{N+2m\}$, we get $p(A,B,I)=0$ , completing the proof that (4) is equivalent to (1).
The proof that (5) is equivalent to (1) is almost identical.
	
	\end{proof}
\begin{lem}	
If $A,B$ commute then $zu_1u_1^*$ commutes with $u_1zu_1^*$ and, thus, $\mathcal{C}_{00}$ is commutative.
\end{lem}
\begin{proof}
	It follows from Lemma~\ref{AB} that, for every $k>N+1$, $P_kS_1S_1^*ZP_k$ commutes with $P_k S_1 Z S_1^*P_k$ if $A$ and $B$ commutes. Thus $zu_1u_1^*$ commutes with $u_1zu_1^*$.	
	
	 
\end{proof}
\begin{lem}\label{isomAB} Let $A,B$ be as above and $\tilde{A},\tilde{B}$ be two selfadjoint matrices in $M_n(\mathbb{C})$ for some $n\geq 1$ such that, for every polynomial $p(x,y,z)$,
\begin{equation}\label{tilde}
p(A,B,I)=p(B,A,I)=0 \iff p(\tilde{A},\tilde{B},I)=p(\tilde{B},\tilde{A},I)=0.
\end{equation}	
(Possibly $A=\tilde{A}$ and $B=\tilde{B}$).	
	 Then the following are equivalent.
	\begin{enumerate}
		\item [(1)] There is a (unital) isomorphism $\phi:C^*(\tilde{A},\tilde{B},I) \rightarrow \mathcal{C}_{00}$ such that $\phi(\tilde{A})=u_1u_1^*z$ and $\phi(\tilde{B})=u_1zu_1^*$.
		\item[(2)]  There is a (unital) isomorphism $\phi:C^*(\tilde{A},\tilde{B},I) \rightarrow \mathcal{C}_{00}$ such that $\phi(\tilde{B})=u_1u_1^*z$ and $\phi(\tilde{A})=u_1zu_1^*$.
		\item[(3)]  For every polynomial $p(x,y,z)$, $p(\tilde{A},\tilde{B},I)=0$ if and only if $p(\tilde{B},\tilde{A},I)=0$.
	\end{enumerate}
\end{lem}
\begin{proof} We prove the equivalence of (1) and (3). The proof of the equivalence of (2) and (3) is similar (or follows by symmetry).
	 
To prove that (1) implies (3), assume that (1) holds. If $p(\tilde{A},\tilde{B},I)=0$ then it follows from (1) that $p(u_1u_1^*z,u_1zu_1^*,u_1u_1^*)=0$. But, using the equivalence of (4) and (5) in Lemma~\ref{AB}, we get $p(u_1zu_1^*,u_1u_1^*z,u_1u_1^*)=0$ and, using (1) again, we get $p(\tilde{B},\tilde{A},I)=0$. A similar argument proves the converse.	

Now assume (3). Since $C^*(\tilde{A},\tilde{B},I)$ is finite dimensional and $\tilde{A},\tilde{B}$ are selfadjoint, what we need to show is that such $\phi$ is well defined and injective (surjectivity is clear). But this means that, for every polynomial $p$, $p(\tilde{A},\tilde{B},I)=0$ (and, therefore, by (3), also $p(\tilde{B},\tilde{A},I)=0$) if and only if $p(u_1u_1^*z,u_1zu_1^*,I)=0$.  But, $p(\tilde{A},\tilde{B},I)=0$ if and only if $p(\tilde{A},\tilde{B},I)=p(\tilde{B},\tilde{A},I)=0$ (by (3)) and, by (\ref{tilde}), this is equivalent to $p(A,B,I)=p(B,A,I)=0$. By  Lemma~\ref{AB}, this is equivalent to  $p(u_1u_1^*z,u_1zu_1^*,I)=0$.
	
	\end{proof}

\begin{cor}\label{findim}
	In the setup of this section, $\mathcal{C}_{00}$ is a finite dimensional $C^*$-algebra.
\end{cor}
\begin{proof}
Take $\tilde{A}=A\oplus B$ and $\tilde{B}=B\oplus A$. They satisfy (3) of Lemma~\ref{isomAB} and it follows that $\mathcal{C}_{00}\cong C^*(\tilde{A},\tilde{B},I)$.	
	\end{proof}

\begin{example}\label{ex1}
Let $N=1$  and $A,B$ the two positive $2\times 2$ matrices 
 $B=\left( \begin{array}{cc} 2 & 1 \\ 1 & 2 \end{array} \right) $ and  $A=\left( \begin{array}{cc} 3 & 0 \\ 0 & 1 \end{array} \right) $ . We will show that Condition (3) of Lemma~\ref{isomAB} holds and $C^*(A,B,I)=M_2$, consequently, $\mathcal{C}_{00}\cong M_2$ and $q(\mathcal{D})\cong M_{2^{\infty}}$ and it has no non trivial ideals.


To see this, set $W=\frac{1}{\sqrt{2}}\left( \begin{array}{cc} 1 & 1  \\ 1 & -1 \end{array} \right) $. Then $W$ is a selfadjoint unitary and a simple computation shows that $WAW^*=B$ and $WBW^*=A$. Thus, for a polynomial $p(x,y,z)$, $Wp(A,B,I)W^*=p(B,A,I)$ and condition (3) of Lemma~\ref{isomAB} follows. Thus, there is a (unital) isomorphism $\phi:C^*(A,B,I) \rightarrow \mathcal{C}_{00}$ such that $\phi(A)=u_1u_1^*z$ and $\phi(B)=u_1zu_1^*$.

Now, write $\mathcal{A}=C^*(A,B,I)\subseteq M_2$ and let $e_{i,j}$ be the matrix units in $M_2$. Then $e_{1,1}=\frac{1}{2}(A-I)\in \mathcal{A}$, $e_{2,2}=I-e_{1,1}\in \mathcal{A}$, $e_{1,2}=e_{1,1}(B-2I)\in \mathcal{A}$ and $e_{2,1}=e_{1,2}^*$. Therefore $C^*(A,B,I)=M_2$.
\end{example}

\begin{example}\label{ex2}
Let $N=1$,  $A=\left( \begin{array}{cc} 1 & 0 \\ 0 & 2 \end{array} \right) $ and  $B=\left( \begin{array}{cc} 2 & 0 \\ 0 & 1 \end{array} \right) $ . Condition (3) of Lemma~\ref{isomAB} clearly holds and, thus, $\mathcal{C}_{00}\cong \mathbb{C}^2$. In fact,  there is a (unital) isomorphism $\phi:C^*(A,B,I) \rightarrow \mathcal{C}_{00}$ such that $\phi(A)=u_1u_1^*z$ and $\phi(B)=u_1zu_1^*$.  Consequently, $q(\mathcal{D})\cong M_{2^{\infty}}\otimes \mathbb{C}^2$ and it has $2$ non trivial ideals. We will show that these ideals are not invariant, so, it does not have a non trivial invariant ideals.

Write $\{e_1,e_2\}$ for the standard basis of $\mathbb{C}^2$. Then $\mathbb{C}^2$ has 2 non trivial ideals $J_1=\mathbb{C}e_1$ and $J_2=\mathbb{C}e_2$. Since $\phi(e_1)=\phi(-\frac{1}{3}A+\frac{2}{3}B)=-\frac{1}{3}u_1u_1^*z+\frac{2}{3}u_1zu_1^*$ and  $\phi(e_2)=\phi(\frac{2}{3}A-\frac{1}{3}B)=\frac{2}{3}u_1u_1^*z-\frac{1}{3}u_1zu_1^*$, the corresponding ideals in $M_2\otimes \mathcal{C}_{00}$ are $M_2\otimes (2u_1zu_1^*-u_1u_1^*z)$ and $M_2\otimes (2u_1u_1^*z-u_1zu_1^*)$. By applying $\tau^{-1}$, we get, in $\mathcal{C}_{0}$, the ideals
$$J_1=\tau^{-1}(M_2\otimes (2u_1zu_1^*-u_1u_1^*z))=span_{k,l} \tau^{-1}(e_{k,l}\otimes (2u_1zu_1^*-u_1u_1^*z))$$  $$=span_{k,l} u_ku_1^*(2u_1zu_1^*-u_1u_1^*z)u_1u_l^* = span_{k,l} (2u_k z u_l^*-u_ku_l^*z) $$ and  
$$J_2=\tau^{-1}(M_2\otimes (2u_1u_1^*z-u_1zu_1^*))=span_{k,l} \tau^{-1}(e_{k,l}\otimes (2u_1u_1^*z-u_1zu_1^*))$$  $$=span_{k,l} u_ku_1^*(2u_1u_1^*z-u_1zu_1^*)u_1u_l^* = span_{k,l} (2u_k  u_l^*z-u_kzu_l^*) $$.
Now, fix $1\leq m,j \leq 2$ and consider $u_m^*J_iu_j$. Note that $u_m^*\mathcal{C}_0 u_j\subseteq \mathcal{C}_0$. In fact, $u_m^*u_ku_l^*u_j=\delta_{k,m}\delta_{j,l}I$, $u_m^*u_ku_l^*zu_j=\delta_{k,m}u_l^*zu_j=\delta_{k,m}\Sigma_{n,t}u_l^*u_tu_nzu_n^*u_t^*u_j=\delta_{k,m} \delta_{j,l} \Sigma_n u_nzu_n^*$ and $u_m^*u_kzu_l^*u_j=\delta_{m,k}\delta_{j,l} z$.

Thus $$u_m^*J_1u_j=span_{k,l}(2\delta_{k,m} \delta_{j,l} \Sigma_n u_nzu_n^*-\delta_{m,k}\delta_{j,l} z)=\mathbb{C}(2z-\Sigma_n u_nzu_n^*)$$ and
$$u_m^*J_2u_j=span_{k,l}(2\delta_{k,m} \delta_{j,l} \Sigma_n u_nzu_n^*-\delta_{m,k}\delta_{j,l} z)=\mathbb{C}(2\Sigma_n u_nzu_n^* -z).$$
In order to apply $\tau$, we note that  $\tau(z)=I\otimes zu_1u_1^*$ and $\tau(\Sigma_n u_nzu_n^*)=I\otimes u_1zu_1^* $. Thus
$$\tau(u_m^*J_1u_j)=\mathbb{C}(I\otimes(2z u_1u_1^*-u_1zu_1^*))\nsubseteq M_2\otimes (2u_1zu_1^*-zu_1u_1^*)=\tau(J_1)$$
$$\tau(u_m^*J_2u_j)=\mathbb{C}(I\otimes (2u_1zu_1^*-zu_1u_1^*))\nsubseteq M_2\otimes (2u_1u_1^*z-u_1zu_1^*)=\tau(J_2).$$ 
Thus, for every ideal $J$ of $\mathcal{C}_0$ and every $m,j$, $u_m^*Ju_j\nsubseteq J$.

\end{example}

\begin{example}\label{ex3}
	Let $N=1$,  $A=3I $ and  $B=\left( \begin{array}{cc} 3 & 0 \\ 0 & 1 \end{array} \right) $ . Condition (3) of Lemma~\ref{isomAB} does not hold for $A,B$ but it does hold for $\tilde{A},\tilde{B}$ where
	$$ \tilde{A} =\left( \begin{array}{ccc} 3 & 0 & 0 \\ 0 & 3 & 0 \\ 0 & 0 & 1 \end{array} \right), \tilde{B} =\left( \begin{array}{ccc} 3 & 0 & 0 \\ 0 & 1 & 0 \\ 0 & 0 & 3 \end{array} \right) $$  and, thus, $\mathcal{C}_{00}\cong \mathbb{C}^3$. In fact,  there is a (unital) isomorphism $\phi:C^*(\tilde{A},\tilde{B},I) \rightarrow \mathcal{C}_{00}$ such that $\phi(\tilde{A})=u_1u_1^*z$ and $\phi(\tilde{B})=u_1zu_1^*$.  Consequently, $q(\mathcal{D})\cong M_{2^{\infty}}\otimes \mathbb{C}^3$.
	
	Each non trivial ideal in $\mathbb{C}^3$ is $\mathbb{C}e_i+\mathbb{C}e_j$ for some $1\leq i,j \leq 3$. So we now consider the ideals $\mathbb{C}e_i$ in $\mathbb{C}^3$. Since $\phi(e_1)=\frac{1}{2}\phi(\tilde{A}+\tilde{B}-4I)=\frac{1}{2}(u_1u_1^*z+u_1zu_1^*-4u_1u_1^*)$, $\phi(e_2)=\frac{1}{2}(3I-\tilde{B})=\frac{1}{2}(3u_1u_1^*-u_1zu_1^*)$ and $\phi(e_3)=\frac{1}{2}\phi(3I-\tilde{A})=\frac{1}{2}(3u_1u_1^*-u_1u_1^*z)$, the corresponding ideals in $M_2\otimes \mathcal{C}_{00}$ are $M_2\otimes (u_1u_1^*z+u_1zu_1^*-4u_1u_1^*)$, $M_2\otimes (3u_1u_1^*-u_1zu_1^*)$ and $M_2\otimes (3u_1u_1^*-u_1u_1^*z)$. By applying $\tau^{-1}$, we get the corresponding ideals in $\mathcal{C}_0$ and we write
	$$ J_1=\tau^{-1}(M_2\otimes (u_1u_1^*z+u_1zu_1^*-4u_1u_1^*))=span_{k,l} \tau^{-1}(e_{k,l}\otimes  (u_1u_1^*z+u_1zu_1^*-4u_1u_1^*))$$  $$=span_{k,l} u_ku_1^* (u_1u_1^*z+u_1zu_1^*-4u_1u_1^*)u_1u_l^*=span_{k,l}(u_ku_l^*z+u_kzu_l^*-4u_ku_l^*).$$
	$$ J_2=\tau^{-1}(M_2\otimes (3u_1u_1^*-u_1zu_1^*))=span_{k,l} \tau^{-1}(e_{k,l}\otimes  (3u_1u_1^*-u_1zu_1^*))$$  $$=span_{k,l} u_ku_1^* (3u_1u_1^*-u_1zu_1^*)u_1u_l^*=span_{k,l}(3u_ku_l^*-u_kzu_l^*).$$
	$$ J_3=\tau^{-1}(M_2\otimes (3u_1u_1^*-u_1u_1^*z))=span_{k,l} \tau^{-1}(e_{k,l}\otimes  (3u_1u_1^*-u_1u_1^*z))$$  $$=span_{k,l} u_ku_1^* (3u_1u_1^*-u_1u_1^*z)u_1u_l^*=span_{k,l}(3u_ku_l^*-u_ku_l^*z).$$
	Now, fix $1\leq m,j \leq 2$ and consider $u_m^*J_iu_j$. Note that $u_m^*\mathcal{C}_0 u_j\subseteq \mathcal{C}_0$. In fact, $u_m^*u_ku_l^*u_j=\delta_{k,m}\delta_{j,l}I$, $u_m^*u_ku_l^*zu_j=\delta_{k,m}u_l^*zu_j=\delta_{k,m}\Sigma_{n,t}u_l^*u_tu_nzu_n^*u_t^*u_j=\delta_{k,m} \delta_{j,l} \Sigma_n u_nzu_n^*$ and $u_m^*u_kzu_l^*u_j=\delta_{m,k}\delta_{j,l} z$.
	
	Thus, $$u_m^*J_1u_j=span_{k,l}(\delta_{k,m} \delta_{j,l} \Sigma_n u_nzu_n^*+\delta_{m,k}\delta_{j,l} z-\delta_{m,k}\delta_{j,l}I)=\mathbb{C} (  \Sigma_n u_nzu_n^*+ z-4I).$$
	$$u_m^*J_2u_j=span_{k,l}(\delta_{k,m}\delta_{j,l}3I-\delta_{m,k}\delta_{j,l} z)=\mathbb{C}(3I-z).$$
	$$u_m^*J_3u_j=span_{k,l}(\delta_{k,m}\delta_{j,l}3I-\delta_{k,m} \delta_{j,l} \Sigma_n u_nzu_n^*)=\mathbb{C}(3I-\Sigma_n u_nzu_n^*). $$
	In order to apply $\tau$, we note that $\tau(I)=I\otimes u_1u_1^*$, $\tau(z)=I\otimes zu_1u_1^*$ and $\tau(\Sigma_n u_nzu_n^*)=I\otimes u_1zu_1^* $. Thus
	$$\tau(u_m^*J_1u_j)=\mathbb{C}(I\otimes( u_1zu_1^*+zu_1u_1^*-4u_1u_1^*))\subseteq M_2\otimes (u_1u_1^*z+u_1zu_1^*-4u_1u_1^*)$$
	$$\tau(u_m^*J_2u_j)=\mathbb{C}(I\otimes (3u_1u_1^*-zu_1u_1^*))\nsubseteq M_2\otimes (3u_1u_1^*-u_1zu_1^*)$$ and
	$$\tau(u_m^*J_3u_j)=\mathbb{C}(I\otimes (3u_1u_1^*-u_1zu_1^*))\nsubseteq M_2 \otimes (3u_1u_1^*-u_1u_1^*z) .$$
	Thus, for every $m,j$, $u_m^*J_1u_j\subseteq J_1$ but $u_m^*J_2u_j\nsubseteq J_2$ and $u_m^*J_3u_j\nsubseteq J_3$

	\end{example}

\section{Simplicity}
 We assume here that Condition A(p) holds for some $p>0$. 
 In order to discuss the simplicity of the $C^*$-algebra $\mathcal{T}(\mathbb{C}^d,Z)/\mathcal{K}$ in this case, we shall need to analyse the ideals of $q(\cald)$.
 
 


 
 \begin{prop}\label{invariantideals}
 	The following statements are equivalent.
 	\begin{enumerate}
 		\item[(1)] There is a non trivial ideal $J$ in $q(\cald)$ such that, for every $1\leq i,j \leq d$, $u_i^*Ju_j\subseteq J$.
 		\item[(2)] There is a non trivial ideal $J_0$ in $\mathcal{C}_0$ such that  for every $1\leq i,j \leq d$, $u_i^*J_0u_j\subseteq J_0$.
 		\item[(3)] There is a non trivial  ideal $J_{00}$ in $\mathcal{C}_{00}$ such that, for every $1\leq i,j \leq d$, $\tau(u_i^* \tau^{-1}(M_{d^q}\otimes J_{00}) u_j)\subseteq M_{d^q}\otimes  J_{00}$ where $\tau:\mathcal{C}_0 \rightarrow M_{d^q}\otimes \mathcal{C}_{00}$ is the isomorphism (see the discussion preceeding Lemma~\ref{C00}).
 	\end{enumerate}
 \end{prop} 
\begin{proof} We start by assuming (2) and write $J_0$ for a non trivial ideal in $\mathcal{C}_0$ such that  for every $1\leq i,j \leq d$, $u_i^*J_0u_j\subseteq J_0$. Then $\tau(J_0)$ is a non trivial ideal of $M_{d^q}\otimes \mathcal{C}_{00}$. We let $J_{00}$ be the (non trivial) ideal in $\mathcal{C}_{00}$ such that $\tau(J_0)=M_{d^q}\otimes J_{00}$. It is then clear that $J_{00}$ satisfies (3). Conversely, given $J_{00}$ as in (3), we set $J_0:=\tau^{-1}(M_{d^q}\otimes J_{00})$ and get (2). Thus (2) and (3) are equivalent.
	
	So it is enough to prove that (1) and (2) are equivalent.
Assume (1) holds and $J$ is a non trivial ideal in $q(\cald)$ such that, for every $0\leq i,j \leq d$, $u_i^*Ju_j\subseteq J$. Then, for every $n$, $J\cap \mathcal{C}_n$ is an ideal in $\mathcal{C}_n$ such that,  for every $0\leq i,j \leq d$, $u_i^*(J\cap \mathcal{C}_n)u_j\subseteq J\cap \mathcal{C}_n$. Since $J=\overline{\cup_n ( J\cap\mathcal{C}_n) }$, there is some $n$ such that $J\cap \mathcal{C}_n$ is a non trivial ideal of $\mathcal{C}_n$. Fix such $n$ and write $J_n=J\cap \mathcal{C}_n$. 
Since, for every $i,j$, $u_i^*\mathcal{C}_n u_j\subseteq \mathcal{C}_n$ (Lemma~\ref{maps}(5)), we also have $$u_i^*J_n u_j\subseteq J_n.$$ By applying this $pn$ times we find that, for every $\alpha_0 ,\beta_0$ such that $|\alpha_0|=|\beta_0|=pn$, $$u_{\alpha_0}^*J_nu_{\beta_0}\subseteq J_n.$$

 Since $\phi_n:M_{d^{pn}}\otimes \mathcal{C}_0\rightarrow \mathcal{C}_n$ (as in Lemma~\ref{maps}) is an isomorphism, there is a non trivial ideal $J_0$ in $\mathcal{C}_0$ such that 
 $$J_n=\phi_n (M_{d^{pn}}\otimes J_0)=\Sigma_{|\alpha|=|\beta|=pn}u_{\alpha}J_0u_{\beta}^*.$$
 Fix $\alpha_0,\beta_0$ such that $|\alpha_0|=|\beta_0|=pn$ and then  $$u_{\alpha_0}^*J_nu_{\beta_0}=u_{\alpha_0}^*(\Sigma_{|\alpha|=|\beta|=pn}u_{\alpha}J_0u_{\beta}^*)u_{\beta_0}=J_0. $$ But, since $u_{\alpha_0}^*J_nu_{\beta_0}\subseteq J_n\subseteq J$, we find that $J_0\subseteq J$ and, as $J$ is an ideal, \begin{equation}\label{alphajbeta}
 u_{\alpha_0}J_0u_{\beta_0}^*\subseteq J.\end{equation}
 
 We want to show that, for every $i,j$ and $a\in J_0$, $u_i^*au_j\in J_0$. So we fix $i,j$ and $a\in J_0$. Take $\alpha',\beta'$ with $|\alpha'|=|\beta'|=np-1$ and write $\alpha=\alpha' i$ and $\beta=\beta' j$. Then $u_i u_{\alpha} u_i^*au_j u_{\beta}^*u_j^*=u_{i \alpha'}(u_iu_i^*)a(u_j u_j^*)u_{\beta'}^*u_j^*$. Since $u_iu_i^*$ and $u_ju_j^*$ lie in $\mathcal{C}_0$, $(u_iu_i^*)a(u_j u_j^*)\in J_0$ and, thus, using (\ref{alphajbeta}) for $\alpha_0=i\alpha$ and $\beta_0=j\beta$, $u_i u_{\alpha} u_i^*au_j u_{\beta}^*u_j^*\in J$. But, then, from the invariance property of $J$, $u_i^*(u_i u_{\alpha} u_i^*au_j u_{\beta}^*u_j^*)u_j\in J$. But this is  $u_{\alpha} u_i^*au_j u_{\beta}^*=\phi_n(e_{\alpha,\beta}\otimes u_i^*au_j)$. Therefore $u_i^*au_j\in J_0$. Since $i,j$ are arbitrary and $a$ is an arbitrary element of $J_0$, (2) follows.

Now assume (2), that is, there is a non trivial ideal $J_0$ in $\mathcal{C}_0$ such that  for every $0\leq i,j \leq d$, $u_i^*J_0u_j\subseteq J_0$. For every $n\geq 1$, let $J_n$ be the	ideal in $\mathcal{C}_n$ defined by  $J_n:=\phi_n(M_{d^{pn}}\otimes J_0)$. Note first that $J_n\subseteq J_{n+1}$. If $b\in J_n$ then $b=\phi_n(\Sigma_{|\alpha|=|\beta|=pn}e_{\alpha,\beta}\otimes a_{\alpha,\beta})=\Sigma_{|\alpha|=|\beta|=pn}u_{\alpha}a_{\alpha,\beta}u_{\beta}^*$ for  $a_{\alpha,\beta}\in J_0$. But then
$$b=\Sigma_{|\alpha|=|\beta|=pn,|\gamma|=|\delta|=p}u_{\alpha \gamma}u_{\gamma}^*a_{\alpha,\beta}u_{\delta}u_{\beta \delta}^*$$
Applying successively the invariance assumption of $J_0$, we get $u_{\gamma}^*a_{\alpha,\beta}u_{\delta}\in J_0$. Thus 
$$b=\Sigma_{|\alpha|=|\beta|=pn,|\gamma|=|\delta|=p}u_{\alpha \gamma}c_{\alpha,\gamma,\beta,\delta}u_{\beta \delta}^*$$ for some $c_{\alpha,\gamma,\beta,\delta}\in J_0$. It follows that $b\in J_{n+1}$.

Now, we write $J=\overline{\cup_n J_n}$ and $J$ is an ideal in $q(\cald)$. To show that $J$ has the invariance property of (1), it suffices to show that, for every $n,i,j$, $u_i^*J_nu_j\subseteq J_n$.

So, we fix $i,j,n$ and $b\in J_n$. Then $b=\Sigma_{|\alpha|=|\beta|=pn}u_{\alpha}a_{\alpha,\beta}u_{\beta}^*$ for  $a_{\alpha,\beta}\in J_0$ and
$$u_i^*bu_j=\Sigma_{|\alpha|=|\beta|=pn}u_i^*u_{\alpha}a_{\alpha,\beta}u_{\beta}^*u_j.$$ 
We can assume that all $\alpha,\beta$ appearing in this sum satisfy $\alpha=i\alpha'$ and $\beta=j \beta'$ for some $\alpha',\beta'$ with $|\alpha'|=|\beta'|=pn-1$ (otherwise, the corresponding summand is $0$). Thus, we can write $u_i^*u_{\alpha}=u_{\alpha'}$ and $u_{\beta}^*u_j=u_{\beta'}^*$ and get
$$u_i^*bu_j=\Sigma_{|\alpha'|=|\beta'|=pn-1}u_{\alpha'}a_{\alpha,\beta}u_{\beta'}^*=
\Sigma_{|\alpha'|=|\beta'|=pn-1,k,l}u_{\alpha'}u_ku_k^*a_{\alpha,\beta}u_lu_l^*u_{\beta'}^*$$  $$=\Sigma_{|\alpha'|=|\beta'|=pn-1,k,l}u_{\alpha'k}(u_k^*a_{\alpha,\beta}u_l)u_{\beta'l}^*.$$  Since $ u_k^*a_{\alpha,\beta}u_l\in J_0$ (as we assume that (2) holds) and $|\alpha' k|=|\beta' l|=pn$, $u_i^*bu_j\in J_n$ and this proves (1).
	
	\end{proof}

In order to discuss the simplicity of the algebra $\mathcal{T}(\mathbb{C}^d,Z)/\mathcal{K}$ we need first the following definition.

\begin{defn}
Let $X$ be a $C^*$-correspondence over a unital $C^*$-algebra $A$. We say that it is \emph{minimal} if there are no non trivial ideals $J\subseteq A$ such that 
$\langle X, JX \rangle \subseteq J$. It is said to be \emph{nonperiodic} if $X^{\otimes n}$ and $A$ are isometric (that is, there is a unitary map from $X^{\otimes n}$ onto $A$) only if $n=0$. 	
	
	\end{defn}

The following theorem was proved by J. Schweizer \cite[Theorem 3.9]{Sch01}.

\begin{thm}\label{simpleCP}
Let $X$ be a full $C^*$-correspondence over a unital $C^*$-algebra $A$. Then the Cuntz-Pimsner algebra $\mathcal{O}(X,A)$ is simple if and only if $X$ is minimal and nonperiodic.
\end{thm}

Using Proposition~\ref{nonperiodic} ,  we get the following.

\begin{thm}\label{simplicity} If $d>1$ and Condition A(p) holds for some $p>0$,
	then the $C^*$-algebra $\mathcal{T}(\mathbb{C}^d,Z)/\mathcal{K}$ is simple if and only if  the equivalent conditions of \ref{invariantideals} do not holds.
	
	If $d=1$ and Condition A(p) holds (so that $z$ commutes with $u^p$) then the $C^*$-algebra $\mathcal{T}(\mathbb{C}^d,Z)/\mathcal{K}$ is not simple.
\end{thm}

Observe that the case $d=1$ in this theorem follows also from \cite[Proposition 2.2]{G}.


Using our analysis of Example~\ref{ex1}, Example~\ref{ex2} and Example~\ref{ex3} and using Lemma~\ref{simplicityA1}, we get

\begin{cor}\label{EX} Suppose $d>1$.
	The algebras $\mathcal{T}(\mathbb{C}^d,Z)/\mathcal{K}$ of Example~\ref{ex1} and of Example~\ref{ex2} are simple while the algebra of Example~\ref{ex3} is not.
	
	Also, when Condition A(1) holds, the algebra $\mathcal{T}(\mathbb{C}^d,Z)/\mathcal{K}$ is simple if and only if $C^*(z)$ has no non trivial ideals. 
	
	If $z$ is a normal operator, this happens if and only if 
	$z\in \mathbb{C}I$ which is the case if and only if $Z_k\rightarrow \lambda I$ for some $\lambda \in \mathbb{C}$. (Note that, when $Z_k=I$ for all $k$, the algebra is the Cuntz algebra $\mathcal{O}_d$). 
\end{cor}

\section{Representations of $\mathcal{T}/\mathcal{K}$}
The main theorem of this section is the following.
 
\begin{thm}\label{representations}
Suppose Condition A(p) holds, $\{U_i\}_{1\leq i \leq d}$ is a Cuntz family (in $B(H)$), $R\in B(H)$ is a  invertible operator that commutes with $U_{\alpha}$ and with $U_{\alpha}^*$ for all words $\alpha$ of length $p$. Then the following are equivalent.
\begin{enumerate}
                \item [(1)] There is a $^*$-representation $\pi$ of $\mathcal{T}/\mathcal{K}$ on $H$ such that $\pi(u_i)=U_i$ for $1\leq i \leq d$ and $\pi(z)=R$.
                \item[(2)] There is a $^*$-representation $\pi_{00}$ of $\mathcal{C}_{00}$ on $H$ such that (in the notation of Lemma~\ref{C00}), 
                $$\pi_{00}(u_{\xi_m}f(z)u_{\xi_{q-m}}u_{\xi_q}^*)=U_1^mf(R)U_i^{q-m}U_1^{* q}$$ for every polynomial $f$ and $0\leq m\leq q=p-1$.
\end{enumerate}           
Of course, if $z$ is assumed positive (as we can, by Lemma~\ref{ueq}), then we might take $R$ to be positive and then the requirement that it commutes with $U_{\alpha}^*$ is redundant.               
\end{thm}
\begin{proof}
It is clear that (1) implies (2). To prove the other direction we assume that $\{U_i\}$ and $R$ are as in the statement of the theorem and (2) holds. Then, we proceed in two steps.
 
First, we show that there is a $^*$-representation $\sigma$ of $q(\cald)$ on $H$ such that, for the generators (see Lemma~\ref{qDAp}) we have $\sigma(u_{\alpha}z^mu_{\beta}^*)=U_{\alpha}R^mU_{\beta}^*$ for $|\alpha|=|\beta|$.
 
Then, we define a linear map $T:q(F)\rightarrow B(H)$ by $T(\Sigma_i u_i d_i)=\Sigma_i U_i\sigma(d_i)$, for $d_i\in q(\cald)$. (Note that, by Lemma~\ref{qF}, $q(F)=\Sigma u_iq(\cald)$). Then we show that the pair $(\sigma,T)$ induces a $^*$-representation of the Cuntz-Pimsner algebra $\mathcal{O}(q(\cald),q(F))$ (which is isomorphic to $\mathcal{T}/\mathcal{K}$).   

Since $\{U_i\}$ is a Cuntz family, we shall repeatedly use the following (which also hold for $\{u_i\}$ by Lemma~\ref{alpha}).
\begin{enumerate}
	\item [(i)] For words $\alpha,\beta$ of the same length, $U_{\alpha}^*U_{\beta}=\delta_{\alpha,\beta}$
	\item[(ii)] For every $n>0$, $\Sigma_{|\alpha|=n}U_{\alpha}U_{\alpha}^*=I$.
\end{enumerate}
Now we turn to the proof of the first step (the existence of $\sigma$ as above).

We are given a $^*$-representation $\pi_{00}$ of $\mathcal{C}_{00}$ on $H$ such that (in the notation of Lemma~\ref{C00}), 
$$\pi_{00}(u_{\xi_m}f(z)u_{\xi_{q-m}}u_{\xi_q}^*)=U_1^mf(R)U_1^{q-m}U_1^{* q}$$ for every polynomial $f$ and $0\leq m\leq q=p-1$.

Note that, for every $a\in \mathcal{C}_{00}$, $a=u_{\xi_q}u_{\xi_q}^*au_{\xi_q}u_{\xi_q}^*$. Since $\pi_{00}(u_{\xi_q}u_{\xi_q}^*)=U_1^qU_1^{q *}$, we have that  $\pi_{00}$ is in fact a representation of $\mathcal{C}_{00}$ on  $U_1^qU_1^{q *}H$.
Using Lemma~\ref{M2C} (with $u_{\alpha}u_{\beta}^*$ in place of $t_{\alpha,\beta}$) and the discussion following it, we can write every $b\in \mathcal{C}_0$ as
$$b=\tau^{-1}(\Sigma_{|\alpha|=|\beta|=q} e_{\alpha,\beta}\otimes a_{\alpha,\beta})=\Sigma_{|\alpha|=|\beta|=q} u_{\alpha}u_{\xi_q}^*a_{\alpha,\beta}u_{\xi_q}u_{\beta}^* $$ for some $\{a_{\alpha,\beta}: |\alpha|=|\beta|=q \}\subseteq \mathcal{C}_{00}$ and $b=0$ only if $a_{\alpha,\beta}=0$ for all $\alpha,\beta$.
Therefore, we can define $\sigma_0$ on $\mathcal{C}_{0}$ by
\begin{equation}\label{sigma0}
\sigma_0(\Sigma_{|\alpha|=|\beta|=q} u_{\alpha}u_{\xi_q}^*a_{\alpha,\beta}u_{\xi_q}u_{\beta}^* )=\Sigma_{|\alpha|=|\beta|=q} U_{\alpha}U_1^{* q}\pi_{00}(a_{\alpha,\beta})U_1^qU_{\beta}^* .\end{equation} 
Then $\sigma_0$ is well defined (as $\Sigma_{|\alpha|=|\beta|=q} u_{\alpha}u_{\xi_q}^*a_{\alpha,\beta}u_{\xi_q}u_{\beta}^* =0$ only if all $a_{\alpha,\beta}$ are $0$) and is clearly selfadjoint. To show that it is a $^*$-representation, it is left to prove that it is multiplicative. But
$$bb'=( \Sigma_{|\alpha|=|\beta|=q}u_{\alpha}u_{\xi_q}^*a_{\alpha,\beta}u_{\xi_q}u_{\beta}^* )(\Sigma_{|\theta|=|\gamma|=q} u_{\theta}u_{\xi_q}^*a_{\theta,\gamma}'u_{\xi_q}u_{\gamma}^* )= \Sigma_{|\alpha|=|\beta|=|\gamma|=q}u_{\alpha}u_{\xi_q}^*a_{\alpha,\beta}a_{\beta,\gamma}'u_{\xi_q}u_{\gamma}^*$$ 
since $u_{\xi_q}u_{\beta}^* u_{\theta}u_{\xi_q}^*=\delta_{\theta,\beta}u_{\xi_q}u_{\xi_q}^*$ and $u_{\xi_q}u_{\xi_q}^*a=a$ for every $a\in \mathcal{C}_{00}$. Applying $\sigma_0$ we get

$$\sigma_0(bb')=\Sigma_{|\alpha|=|\beta|=|\gamma|=q}U_{\alpha}U_1^{* q}\pi_{00}(a_{\alpha,\beta}a_{\beta,\gamma}')U_1^qU_{\gamma}^*. $$
 But, since  $U_1^qU_{\beta}^* U_{\theta}U_1^{* q}=\delta_{\theta,\beta}U_1^qU_1^{* q}$ and $U_1^qU_i^{* q}\pi_{00}(a)=\pi(a)$ for every $a\in \mathcal{C}_{00}$, we get

$$\sigma_0(bb')=( \Sigma_{|\alpha|=|\beta|=q}U_{\alpha}U_1^{* q}\pi_{00}(a_{\alpha,\beta})U_1^qU_{\beta}^* )(\Sigma_{|\theta|=|\gamma|=q} U_{\theta}U_1^{* q}\pi_{00}(a_{\theta,\gamma}')U_1^qU_{\gamma}^* )= \sigma_0(b)\sigma_0(b').$$ 
Thus $\sigma_0$ is a representation of $\mathcal{C}_0$ on $H$. Note that, for every $a\in \mathcal{C}_{00}$ and $|\alpha|=|\beta|=q$, we have
\begin{equation}\label{sigma01}
\sigma_0(u_{\alpha}u_{\xi_q}^*au_{\xi_q}u_{\beta}^* )=U_{\alpha}U_1^{* q}\pi_{00}(a)U_1^qU_{\beta}. \end{equation}
Thus, for $l\geq 0$ and $0\leq |\alpha|=|\beta|<p$, \begin{equation}\label{sigma02}
\sigma_0(u_{\alpha}z^lu_{\beta}^*)=\sigma_0(u_{\alpha}u_{\xi_q}^*u_{\xi_q}z^lu_{\xi_q}^*u_{\xi_q}u_{\beta}^*)=U_{\alpha}U_1^{* q}(U_1^qR^lU_1^{* q})U_1^qU_{\beta}=U_{\alpha}R^lU_{\beta}^*.\end{equation} 
In particular $$\sigma_0(z)=R.$$

Next, for every $n>0$ we define a $^*$-representation of $\mathcal{C}_n$ on $H$. To do this, note that it follows from Lemma~\ref{maps} (3) that every $b\in \mathcal{C}_n$ can be written uniquely as
$$b=\Sigma_{|\alpha|=|\beta|=pn}u_{\alpha}a_{\alpha,\beta}u_{\beta}^*$$
for some $\{a_{\alpha,\beta}\}$ in $\mathcal{C_0}$.
Thus, we define $\sigma_n$ on $\mathcal{C}_n$ by
$$\sigma_n(\Sigma_{|\alpha|=|\beta|=pn}u_{\alpha}a_{\alpha,\beta}u_{\beta}^*)=\Sigma_{|\alpha|=|\beta|=pn}U_{\alpha}\sigma_0(a_{\alpha,\beta})U_{\beta}^*.$$

It is clear that $\sigma_n$ is well defined and selfadjoint. The proof that it is multiplicative proceeds along the same lines as the proof above that $\sigma_0$ is multiplicative and is omitted.

Recall (Equation~\ref{Cn}) that $\mathcal{C}_n$ is generated by $\{u_{\alpha}z^lu_{\beta}^*:\; l\geq 0,\; np\leq |\alpha|=|\beta|<(n+1)p\;\}$. Given $\alpha,\beta$ with $np\leq |\alpha|=|\beta|<(n+1)p$, we can write $u_{\alpha}=u_{\alpha'}u_{\gamma}, u_{\beta}=u_{\beta'}u_{\theta}$ where $|\alpha'|=|\beta'|=np$ and $0\leq |\gamma|=|\theta|<p$. It follows that $u_{\gamma}z^lu_{\theta}^*\in \mathcal{C}_0$ and $u_{\alpha}z^lu_{\beta}^*=u_{\alpha'}u_{\gamma}z^lu_{\theta}^*u_{\beta'}^*$.

Thus $\sigma_n(u_{\alpha}z^lu_{\beta}^*)=\sigma_n(u_{\alpha'}(u_{\gamma}z^lu_{\theta}^*)u_{\beta'}^*)=U_{\alpha'}\sigma_0(u_{\gamma}z^lu_{\theta}^*)U_{\beta'}^*$ and, using Equation~\ref{sigma02}, we have, for $np\leq |\alpha|=|\beta|<(n+1)p$,
\begin{equation}\label{sigman1}
\sigma_n(u_{\alpha}z^lu_{\beta}^*)=U_{\alpha'}U_{\gamma}R^lU_{\theta}^*U_{\beta'}^*=U_{\alpha}R^lU_{\beta}^*.
\end{equation}


Recall that, for every $n\geq 0$, $\mathcal{C}_n\subseteq \mathcal{C}_{n+1}$ and $q(\cald)=\overline{\cup \mathcal{C}_n}$. In order to use the representations $\{\sigma_n\}$ to define a representation $\sigma$ on $q(\cald)$, we need to show that, for every $n \geq 0$, $\sigma_{n+1}$, restricted to $\mathcal{C}_n$, equals $\sigma_n$.

In Lemma~\ref{maps} (1), we denoted the inclusion of $\mathcal{C}_n$ in $\mathcal{C}_{n+1}$, $\iota_n$. Thus, we want to show that $\sigma_{n+1}\circ \iota_n=\sigma_n$. So, fix $u_{\alpha}z^lu_{\beta}^*\in \mathcal{C}_n$ (with $np\leq |\alpha|=|\beta|<(n+1)p$). Then $\iota_n(u_{\alpha}z^lu_{\beta}^*)=\Sigma_{|\gamma|=p}u_{\alpha}u_{\gamma}z^lu_{\gamma}^*u_{\beta}^*$ and
$$\sigma_{n+1}\circ \iota_n (u_{\alpha}z^lu_{\beta}^*)=\sigma_{n+1}(\Sigma_{|\gamma|=p}u_{\alpha \gamma}z^lu_{\beta \gamma}^*)=\Sigma_{|\gamma|=p}U_{\alpha \gamma}R^lU_{\beta \gamma}^*=\Sigma_{|\gamma|=p}U_{\alpha}U_{ \gamma}R^lU_{ \gamma}^*U_{\beta}^*$$  $$=\Sigma_{|\gamma|=p}U_{\alpha}U_{\gamma}U_{\gamma}^*R^lU_{\beta}^*=U_{\alpha}R^lU_{\beta}^*=\sigma_n(u_{\alpha}z^lu_{\beta}^*)$$ where we used the assumption that $R$ commutes with $U_{\gamma}$ for $|\gamma|=p$.

Thus $\{\sigma_n\}$ defines a unital $^*$-representation $\sigma$ of $q(\cald)=\overline{\cup_n \mathcal{C}_n}$ that satisfies
\begin{equation}\label{sigma}
\sigma(u_{\alpha}z^lu_{\beta}^*)=U_{\alpha}R^lU_{\beta}^* \end{equation} 
for $l\geq 0$ and $0\leq |\alpha|=|\beta|$.

Now we show that the pair $(\sigma,T)$, as defined above, induces a unital $^*$-representation $\pi$ as in (1).

 Thus, we have to show that $(\sigma, T)$ satisfies three conditions, listed in the discussion preceding  Lemma ~\ref{homomorphism}: 1) $T:q(F)\rightarrow B(H)$ is a bimodule map (over $\sigma$), 2) $T$ preserves the inner product: $T(\xi)^*T(\eta)=\sigma(\langle \xi,\eta\rangle)$, and 3) $(\sigma, T)$ satisfies the covariance condition.

For the bimodule property note that it is immediate that $T$ is a right module map. For the left module map property we have to show that given $d\in q(\cald)$ and $\xi\in q(F)$, then    
$T(\varphi_{q(F)}(d)\xi)=\sigma(d)T(\xi)$.
Since every $\xi\in q(F)$ has the form $\sum_iu_id_i$, $d_i\in q(\cald)$, we can reduce the proof to showing that $T(\varphi_{q(F)}(d)u_id_i))=\sigma(d)T(u_id_i)$ for each fixed $i$. Using $\sum_ju_ju_j^*=I$, we obtaini
$$T(\varphi_{q(F)}(d)u_id_i)=T(\sum_j u_ju_j^*(\varphi_{q(F)}(d)u_i)d_i)=\sum_jU_j\sigma(u_j^*du_i)\sigma(d_i).$$
We shall prove that for every $i,j$ and every $d\in q(\cald)$,   $$\sigma(u_j^*du_i)=U_j^*(\sigma(d))U_j.$$ 
Then we obtain
$$\sum_jU_j\sigma(u_j^*(\varphi_{q(F)}(d)u_i)\sigma(d_i)=\sum_jU_jU_j^*\sigma(d)U_i\sigma(d_i)=\sigma(d)U_i\sigma(d_i).$$
Hence,
$T(\varphi_{q(F)}(d)u_id_i)=\sigma(d)T(u_id_i)$ for every $d\in q(\cald)$ and every fixed $i$, and this will proves that 
$$T(\varphi_{q(F)}(d)\xi)=\sigma(d)T(\xi).$$
Thus we have to prove the equality $\sigma(u_j^*du_i)=U_j^*(\sigma(d))U_j$.
By Lemma ~\ref{qDAp}, it is enough to prove it for $d=u_\alpha z^mu_\beta^*$, where $|\alpha|=|\beta|$. We distinguish two cases: $|\alpha|=|\beta|\geq 1$ and $|\alpha|=|\beta|=0$. If $|\alpha|=|\beta|\geq 1$ then we write $\alpha=\alpha_1\alpha'$ and $\beta=\beta_1\beta'$, with $|\alpha'|=|\beta'|=|\alpha|-1$, and we get
$\sigma(u_j^*u_\alpha z^mu_\beta^*u_i)=\sigma(u_j^*u_{\alpha_1} u_{\alpha'}z^mu_{\beta'}^*u_{\beta_1}^*u_i)=\delta_{\alpha_1,j}\delta_{i,\beta_1}\sigma(u_{\alpha'}z^mu_{\beta'}^*)$, 
since $u_j^*u_{\alpha_1}=\delta_{\alpha_1,j}$ and $u_{\beta_1}^*u_i=\delta_{i,\beta_1}$. 
Hence $$\sigma(u_j^*u_\alpha z^mu_\beta^*u_i)=\delta_{\alpha_1,j}\delta_{i,\beta_1}\sigma(u_{\alpha'}z^mu_{\beta'}^*)=\delta_{\alpha_1,j}\delta_{i,\beta_1}U_{\alpha'}R^mU_{\beta'}^*=U_j^*U_{\alpha}R^mU_{\beta}^*U_i$$   $$=U_j^*\sigma(u_{\alpha}z^mu_{\beta}^*)U_i.$$

Now, let $|\alpha|=|\beta|=0$. Then we have to show that $\sigma(u_j^*z^mu_i)=U_j^*\sigma(z^m)U_j$. 
As in the proof of Lemma~\ref{maps}, (5), we can write
$$u_j^*z^mu_i=\sum_{|\gamma|=q}u_\gamma u_\gamma^*u_j^*z^mu_i=
\sum_{|\gamma|=q}u_\gamma z^mu_\gamma^*u_j^*u_i=\delta_{i,j}\sum_{|\gamma|=q}u_\gamma z^mu_\gamma^*.$$
Then, $\sigma(u_j^*z^mu_i)=\sigma(\delta_{i,j}\sum_{|\gamma|=q}u_\gamma z^mu_\gamma^*)=\delta_{i,j}\sum_{|\gamma|=q}U_\gamma R^m U_\gamma^*$. But the same calculations in $B(H)$ in the reverse direction give us: 
$$\delta_{i,j}\sum_{|\gamma|=q}U_\gamma R^m U_\gamma^*=\sum_{|\gamma|=q}U_\gamma R^m U_\gamma^*U_j^*U_i=\sum_{|\gamma|=q}U_\gamma U_\gamma^*U_j^*R^mU_i=U_j^*R^mU_i=U_j^*\sigma(z^m)U_i.$$
Hence, $\sigma(u_j^*z^mu_i)=U_j^*\sigma(z^m)U_j$, and this proves that for every $d\in q(\cald)$ and every $\alpha, \beta$ with $|\alpha|=|\beta|$, holds $\sigma(u_\alpha^* d u_\beta)=U_\alpha^* \sigma(d) U_\beta$.

For 2), let $\xi=\sum_i u_id_i\in q(F)$ and $\eta=\sum_i u_ic_i\in q(F)$, $d_i,c_i\in q(\cald)$. Then
$T(\xi)^*T(\eta)=T(\sum_i u_id_i)^*T(\sum_i u_ic_i)=(\sum_i U_i\sigma(d_i))^*\sum_i U_i\sigma(c_i)=\sum_{i,j}\sigma(d_i)^*U_i^*U_j\sigma(c_i)$. Since $U_i^*U_j=\delta_{i,j}I$, we get 
$$T(\xi)^*T(\eta)=\sum_{i}\sigma_i(d_i)^*\sigma(c_i)=\sigma(\sum_{i}d_i^*c_i).$$ 
But  $\sigma(\sum_{i}d_i^*c_i)=\sigma(\sum_{i,j}d_i^*u_i^*u_jc_j))=\sigma((\sum_i u_id_i)^*\sum_j u_jc_j)=\sigma(\langle\sum_iu_id_i,\sum_iu_ic_i\rangle)=\sigma(\langle\xi,\eta \rangle)$.

For 3), we have seen in Lemma~\ref{homomorphism} that $J_{q(F)}=q(\cald)$. Thus, for $d\in q(\cald)$, 
$\sigma^{(1)}(\varphi_{q(F)}(d))=\sigma^{(1)}(\varphi_{q(F)}(d\sum_i u_iu_i^*))=\sigma^{(1)}(\sum_i\varphi_{q(F)}((d u_i)u_i^*))=\sigma^{(1)}(\sum_i\theta_{du_i,u_i})=
\sum_iT(\varphi_{q(F)}(d)u_i)T(u_i)^*=\sum_i\sigma(d)T(u_i)T(u_i)^*=\sigma(d)\sum_iU_iU_i^*=\sigma(d)$ where, in the second equality, we used Lemma~\ref{qF1} (2).

Hence, the pair $(\sigma, T)$ defines a $*$-representation of $\mathcal{T}/\mathcal{K}$ on $B(H)$, which is obviusly is unital since $\pi(I)=\sigma(I)=I$ and the proof is complete.

   \end{proof}

In the following corollaries we demonstrate how this theorem can be applied by considering several examples.

\begin{cor}\label{repA1}
Suppose Condition A(1) holds and $z\geq 0$ (and then, by Lemma~\ref{A1}, $\mathcal{C}_{00}=\mathcal{C}_0=C^*(z)\cong C(sp(z))$). Then the $^*$-representations of $\mathcal{T}/\mathcal{K}$ are given by a Cuntz family $\{U_i\}_{i=1}^d$ (in $B(H)$) and a positive invertible operator $R\in B(H)$ that commutes with every $U_i$ and satisfies $sp(R)\subseteq sp(z)$.		
	\end{cor}
\begin{proof}
This follows from Theorem~\ref{representations} together with the fact that a $^*$-representation of $C^*(z)\cong C(sp(z))$ is given by sending $z$ to a positive  operator whose spectrum is contained in $sp(z)$.	
	\end{proof}

Note that, for the unweighted case, $z=I$ and $\mathcal{T}/\mathcal{K}\cong \mathcal{O}_d$ and Corollary~\ref{repA1} reduces to the well known description of the representations of $\mathcal{O}_d$.

\begin{cor}\label{repex1}
Let $Z$ be as in Example~\ref{ex1}. Then the $^*$-representations of $\mathcal{T}/\mathcal{K}$ are given by a Cuntz family $\{U_1,U_2\}$ (in $B(H)$) and a selfadjoint invertible operator $R\in B(H)$ that commutes with $U_iU_j$ for all $1\leq i,j \leq 2$ and such that the family $\{f_{i,j}\}_{i,j=1}^2$ defined by $f_{1,1}:=\frac{1}{2}(U_1U_1^*R-U_1U_1^*)$,  $f_{2,2}:=U_1U_1^*-f_{1,1}$,  $f_{1,2}:=f_{1,1}(U_1RU_1^*-2U_1U_1^*)$ and  $f_{2,1}:=f_{1,2}^*$ is a family of $2\times 2$ matrix units in $B(U_1U_1^*H)$.
\end{cor}
\begin{proof}
By Theorem~\ref{representations}, the representations are given by a Cuntz family $\{U_1,U_2\}$ (in $B(H)$) and a selfadjoint invertible operator $R\in B(H)$ that commutes with $U_iU_j$ for all $1\leq i,j \leq 2$ and gives rise to a $^*$-representation $\pi_{00}$ of $\mathcal{C}_{00}$. 	But, as shown in that example, $\mathcal{C}_{00}$ is isomorphic (via $\phi$) to the algebra $C^*(A,B,I)$ where $A,B$ are as in Example~\ref{ex1} . The isomorphism $\phi$ maps $A$ to $u_1u_1^*z$ and $B$ to $u_1zu_1^*$. Thus the representation $\pi_{00}\circ \phi$ (on $U_1U_1^*H$) maps $A$ to $U_1U_1^*R$ and $B$ to $U_1Ru_1^*$. 
It follows that $\pi_{00}\circ \phi$ maps $e_{1,1}$ to $f_{1,1}:=\frac{1}{2}(U_1U_1^*R-U_1U_1^*)$, $e_{2,2}$ into $f_{2,2}:=U_1U_1^*-f_{1,1}$, $e_{1,2}$ into $f_{1,2}:=f_{1,1}(U_1RU_1^*-2U_1U_1^*)$ and $e_{2,1}$ into $f_{2,1}:=f_{1,2}^*$. Thus, $\pi_{00}$ is well defined if and only if $\{f_{i,j}\}$ is a family of $2\times 2$ matrix units in $B(U_1U_1^*H)$.
	
	\end{proof}
\begin{cor}\label{repex2}
	Let $Z$ be as in Example~\ref{ex2}. Then the $^*$-representations of $\mathcal{T}/\mathcal{K}$ are given by a Cuntz family $\{U_1,U_2\}$ (in $B(H)$) and a selfadjoint invertible operator $R\in B(H)$ that commutes with $U_iU_j$ for all $1\leq i,j \leq 2$ and satisfy \begin{enumerate}
		\item [(i)] $sp(U_1U_1^*R)\subseteq \{1,2\}$ (as an operator on $U_1U_1^*H$).
		\item [(ii)] $sp(U_1RU_1^*)\subseteq \{1,2\}$ (as an operator on $U_1U_1^*H$).
		\item[(iii)] $U_1U_1^*R+U_1RU_1^*=3U_1U_1^*$	
	\end{enumerate}
 \end{cor} 
\begin{proof}
To analyze the 	representations of $\mathcal{C}_{00}$ (which is isomorphic to the algebra $\mathbb{C}^2$ in this example) we write $\pi_{00}$ for such a representation and get that $T_1:=\pi_{00}(\phi(e_1))=-\frac{1}{3}U_1U_1^*R +\frac{2}{3}U_1RU_1^*$ and $T_2:=\pi_{00}(\phi(e_2))=\frac{2}{3}U_1U_1^*R -\frac{1}{3}U_1RU_1^*$ where $e_1,e_2$ are the generators of $\mathbb{C}^2$. Then $T_1, T_2$ are projections with $T_1T_2=0$ and $T_1+T_2=U_1U_1^*$. We have $U_1U_1^*R=2T_1+T_2$ and $U_1RU_1^*=2T_1+T_2$ so that (i) and (ii) are clear. Since also $U_1U_1^*R+U_1RU_1^*=3T_1+3T_2= 3U_1U_1^*$, we are done.

	\end{proof}

\end{document}